\documentclass[11pt,reqno]{amsart}

\newtheorem{proposition}{Proposition}[section]

\newtheorem{corollary}[proposition]{Corollary}
\newtheorem{theorem}[proposition]{Theorem}

\theoremstyle{definition}
\newtheorem{definition}[proposition]{Definition}

\newtheorem{examples}[proposition]{Examples}
\newtheorem{remark}[proposition]{Remark}

\newcommand{\thlabel}[1]{\label{th:#1}}
\newcommand{\thref}[1]{Theorem~\ref{th:#1}}
\newcommand{\selabel}[1]{\label{se:#1}}

\newcommand{\prlabel}[1]{\label{pr:#1}}
\newcommand{\prref}[1]{Proposition~\ref{pr:#1}}
\newcommand{\colabel}[1]{\label{co:#1}}
\newcommand{\coref}[1]{Corollary~\ref{co:#1}}

\newcommand{\exlabel}[1]{\label{ex:#1}}
\newcommand{\exref}[1]{Example~\ref{ex:#1}}
\newcommand{\delabel}[1]{\label{de:#1}}
\newcommand{\deref}[1]{Definition~\ref{de:#1}}
\newcommand{\eqlabel}[1]{\label{eq:#1}}
\newcommand{\equref}[1]{(\ref{eq:#1})}

\def\ot{\otimes}

\newcommand{\Cc}{\mathcal{C}}

\def\*C{{}^*\hspace*{-1pt}{\Cc}}
\def\text#1{{\rm {\rm #1}}}

\input xy
\xyoption {all} \CompileMatrices

\usepackage{color,amssymb,graphicx}

\begin{document}

\title[Crossed product of Hopf algebras]
{Crossed product of Hopf algebras}

\author{A. L. Agore}
\address{Faculty of Engineering, Vrije Universiteit Brussel, Pleinlaan 2, B-1050 Brussels, Belgium}
\email{ana.agore@vub.ac.be and ana.agore@gmail.com}



\begin{abstract} The main properties of the crossed product in the category
of Hopf algebras are investigated. Let $A$ and $H$ be two Hopf
algebras connected by two morphism of coalgebras $\triangleright :
H\ot A \to A$, $f:H\ot H\to A$. The crossed product $A
\#_{f}^{\triangleright} \, H$ is a new Hopf algebra containing $A$
as a normal Hopf subalgebra. Furthermore, a Hopf algebra $E$ is
isomorphic as a Hopf algebra to a crossed product of Hopf algebras
$A \#_{f}^{\triangleright} \, H$ if and only if $E$ factorizes
through a normal Hopf subalgebra $A$ and a subcoalgebra $H$ such
that $1_{E} \in H$. The universality of the construction, the
existence of integrals, commutativity or involutivity of the
crossed product are studied. Looking at the quantum side of the
construction we shall give necessary and sufficient conditions for
a crossed product to be a coquasitriangular Hopf algebra. In
particular, all braided structures on the monoidal category of $A
\#_{f}^{\triangleright} \, H$-comodules are explicitly described
in terms of their components. As an example, the braidings on a
crossed product between $H_{4}$ and $k[C_{3}]$ are described in
detail.
\end{abstract}
\maketitle

\section*{Introduction}
The crossed product is a fundamental construction in mathematics.
It was first introduced in group theory related to the famous
extension problem of Holder: any extension $(E, i, \pi)$ of a
group $H$ by a group $G$ is equivalent to a crossed product
extension $(H \#^{\alpha}_{f} \, G, i_H, \pi_G)$. The construction
of crossed products of groups has served as a model for later
generalizations at the level of groups acting on rings
\cite{pasman}, Hopf algebras acting on k-algebras \cite{BCM}, von
Neumann algebras \cite{nakagami}, etc.

The crossed product $A \#^{\triangleright}_{f} \, H$ of a Hopf
algebra $H$ acting on a $k$-algebra $A$ was introduced
independently in \cite{BCM} and \cite{DT} as a generalization of
the crossed product of groups acting on $k$-algebras. It has only
an algebra structure and was studied mainly as an algebra
extension of $A$, being an essential tool in Hopf-Galois
extensions theory as it is well known that Hopf-Galois extensions
with normal basis are equivalent to crossed products with
invertible cocycle. Many algebraic properties of the crossed
product of a Hopf algebra $H$ acting on a $k$-algebra $A$ such as
semisimplicity, semiprimeness, etc. were studied in this setting
(\cite{Chin}, \cite{Quinn}). If, in addition, $A$ and $H$ are Hopf
algebras and the cocycle $f:H\ot H\to A$ and the action
$\triangleright : H\ot A \to A$ are coalgebra maps satisfying two
compatibility conditions then it was proved in \cite[Example 2.5,
2)]{am1} that the crossed product $A \#^{\triangleright}_{f} \, H$
has a natural Hopf algebra structure which we call \textit{the
crossed product of Hopf algebras}. An important feature of the
crossed product of Hopf algebras is the following: a Hopf algebra
$E$ is isomorphic as a Hopf algebra to a crossed product of Hopf
algebras if and only if $E$ factorizes through a normal Hopf
subalgebra and a subcoalgebra containing the unit of $E$
(\thref{fact}).

It turns out that the crossed product of Hopf algebras $A
\#^{\triangleright}_{f} \, H$  is also a special case of the so
called ''cocycle bicrossproduct bialgebra'' constructed for
bialgebras by Majid and Soibelman in \cite[Theorem 2.9]{MS} if we
let the cocylce cross coproduct to be the trivial one, i.e. the
coalgebra structure is the tensor product of coalgebras, and the
action $\triangleright$ and the cocycle f to be coalgebra maps. A
remarkable example is the quantum Weyl algebra which was shown to
have a cocycle bicrossproduct bialgebra structure in
\cite[Corollary 3.4]{MS}. Later on, an antipode for the cocycle
bicrossproduct bialgebra was constructed by Andruskiewitsch and
Devoto in \cite{AD} where they study short exact sequences for
quantum groups.

The paper is organized as follows. In Section $1$ we set up the
notations and conventions to be used throughout and introduce some
preliminary results needed later on. First we recall the
construction of the crossed product of two Hopf algebras
(\deref{definitiacr}) presented as a special case of the unified
product introduced recently in \cite{am1}. Basic examples of
crossed product of Hopf algebras are given in \exref{1.2}: in
particular Molnar's smash product of two Hopf algebras \cite{Mo}
is a special case of the crossed product. Moreover, given $(H, G,
f, \triangleright)$ a crossed system of groups (\cite{am2}) we
show that the group algebra of the crossed product of groups $H
\#_{\alpha}^{f} \, G$ can be built out from the group algebras
$kH$ and $kG$ in a natural way as a crossed product of Hopf
algebras (\exref{1.2}).

In Section $2$ we prove the main property that characterizes the
crossed product of Hopf algebras by showing that a Hopf algebra
$E$ is isomorphic to a crossed product of Hopf algebras if and
only if $E$ factorizes through a normal Hopf subalgebra and a
subcoalgebra containing $1_{E}$ (\thref{fact}). Two universal
properties for the crossed product of Hopf algebras are also
proved. In Section $3$ further properties of the crossed product
of Hopf algebras are studied: existence of integrals,
commutativity, semisimplicity and involutivity. For instance, in
\prref{3.1} necessary and sufficient conditions for a crossed
product of Hopf algebras to be commutative are given. In
\prref{3.3} we construct right integrals for the crossed product
of Hopf algebras and, as a consequence, we prove that a crossed
product $A \#^{\triangleright}_{f} \, H$ is semisimple if and only
if both Hopf algebras $A$ and $H$ are semisimple.

In the last section necessary and sufficient conditions for a
crossed product to be a coquasitriangular (braided) Hopf algebras
are given in \thref{4.5}. In particular, all braided structures on
the category of right $A \#_{f}^{\triangleright} \, H$-comodules
are explicitly described in terms of their components. As a
further motivation we recall that coquasitriangular Hopf algebras
provide solutions to the quantum Yang-Baxter equation. In
\exref{exemplulung}, $2)$ an explicit example of a crossed product
between $H_{4}$ and $k[C_{3}]$ is constructed and, using
\thref{4.5}, the braidings on that
 crossed product are described.

\section{Preliminaries}
Throughout this paper k denotes an arbitrary field. Unless
specified otherwise, all algebras, coalgebras, tensor products and
homomorphisms are over $k$. For a coalgebra $C$, we use Sweedler's
$\Sigma$-notation: $\Delta(c) = c_{(1)}\ot c_{(2)}$,
$(I\ot\Delta)\Delta(c) = c_{(1)}\ot c_{(2)}\ot c_{(3)}$, etc with
summation understood. For a $k$-linear map $f: H \ot H \to A$ we
denote $f(g, \, h) = f (g\ot h)$.\\
Recall that a Hopf algebra $H$ is called (co)semisimple if the
underlying (co)algebra structure is (co)semisimple. An element $t$
in a finite-dimensional Hopf algebra $H$ is called right integral
in $H$ if $tx = \varepsilon(x)t$ for all $x \in H$. Maschke's
theorem states that a finite-dimensional Hopf algebra $H$ is
semisimple if and only if $\varepsilon(t) \neq 0$, for some
left/right integral $t \in H$. For all unexplained notations or
definitions we refer the reader to \cite{M} or \cite{Sw}.

\subsection*{Crossed product of Hopf algebras}\selabel{1.1}
Let $H$ be a Hopf algebra, $A$ a $k$-algebra and two $k$-linear
maps $\triangleright: H \otimes A \rightarrow A$, $f: H \otimes H
\rightarrow A$ such that
\begin{eqnarray}
\eqlabel {0} h \triangleright 1_{A} &{=}& \varepsilon_{H}(h)1_{A}\\
\eqlabel{1} 1_H \rhd a
&{=}& a\\ \eqlabel{2} h \rhd (ab) &{=}& (h_{(1)} \rhd a) (h_{(2)} \rhd b)\\
\eqlabel{3} f(h, 1_{H}) &{=}& f(1_{H}, h) =
\varepsilon_{H}(h)1_{A} \eqlabel{4}
\end{eqnarray}
for all $h\in H$, $a$, $b\in A$. The \textit{crossed product }$A
\#_{f}^{\triangleright} \, H$ of $A$ with $H$ is the $k$-module
$A\ot H$ with the multiplication given by
\begin{equation}\eqlabel{001}
(a \# h) \cdot (c \# g):= a (h_{(1)}\triangleright c)
f\bigl(h_{(2)} , g_{(1)}\bigl) \, \# \, h_{(3)}g_{(2)}
\end{equation}
for all $a$, $c\in A$, $h$, $g\in H$, where we denoted $a\ot h$ by
$a\# h$. It can be proved \cite[Lemma 7.1.2]{M} that $A
\#_{f}^{\triangleright} \, H$ is an associative algebra with
identity element $1_A \# 1_H$ if and only if the following two
compatibility conditions hold:
\begin{eqnarray}
\eqlabel{-1} [g_{(1)} \triangleright (h_{(1)} \triangleright a)]
f\bigl(g_{(2)} , \, h_{(2)} \bigl) &{=}& f(g_{(1)}, \, h_{(1)})
\bigl ( (g_{(2)} h_{(2)}) \triangleright a \bigl)\\  \eqlabel{tmc}
\bigl(g_{(1)} \triangleright f(h_{(1)}, \, l_{(1)})\bigl)
f\bigl(g_{(2)}, \, h_{(2)} l_{(2)} \bigl) &{=}& f(g_{(1)}, \,
h_{(1)}) f(g_{(2)}h_{(2)}, \, l)\eqlabel{cc}
\end{eqnarray}
for all $a\in A$, $g$, $h$, $l\in H$. The first compatibility is
called the \textit{twisted module condition} while \equref{cc} is
called the \textit{cocycle condition}.

Assume now that $A$ is also a Hopf algebra and the maps $\rhd :
H\ot A \to A$ and $f: H\ot H \to A $ are coalgebra morphisms. Then
the crossed product $A \#_{f}^{\triangleright} \, H$ has a Hopf
algebra structure with the coalgebra structure given by the tensor
product of coalgebras if and only if the following two
compatibility conditions hold:
\begin{equation}\eqlabel{co1}
g_{(1)} \otimes g_{(2)} \triangleright a = g_{(2)} \otimes g_{(1)}
\triangleright a
\end{equation}
\begin{equation}\eqlabel{co3}
g_{(1)} h_{(1)} \otimes f(g_{(2)}, \, h_{(2)}) = g_{(2)} h_{(2)}
\otimes f(g_{(1)}, \, h_{(1)})
\end{equation}
for all $g$, $h \in H$ and $a$, $b \in A$ (see \cite[Theorem
2.4]{am1} and \cite[Examples 2.5 (2)]{am1}). Remark that both
relations \equref{co1} and \equref{co3} are trivially fulfilled if
$H$ is cocommutative. As a special case by considering
$\triangleleft$ to be the trivial action in \cite[Proposition
2.8]{am1} we obtain that the antipode of the Hopf algebra $A
\#_{f}^{\triangleright} \, H$ is given by:
\begin{equation}\eqlabel{antipod}
S(a \# g) := \Bigl(S_{A}\bigl[f\bigl(S_{H}(g_{(2)}), \,
g_{(3)}\bigl)\bigl] \# S_{H}(g_{(1)})\Bigl) \cdot \bigl(S_{A}(a)
\# 1_{H}\bigl)
\end{equation}
for all $a\in A$ and $g\in H$.

\begin{definition}\delabel{definitiacr}
A quadruple $(A, H, \triangleright, f)$, where $A$ and $H$ are
Hopf algebras and $\triangleright: H \otimes A \rightarrow A$, $f:
H \otimes H \rightarrow A$ are two coalgebra maps such that
relations \equref{0} - \equref{3} and \equref{-1} - \equref{co3}
are fulfilled is called \textit{crossed system of Hopf algebras}.
The corresponding crossed product of Hopf algebras associated to
the crossed system $(A, H, \triangleright, f)$ will be denoted by
$A \#^{\triangleright}_{f} \, H$.
\end{definition}

\begin{examples}\exlabel{1.2}
$1)$ Let $A$, $H$ be two Hopf algebras and $\triangleright$, $f$
be the trivial action respectively the trivial cocycle, that is:
$a \triangleright h = \varepsilon(a)h$ and $f(g , h) =
\varepsilon(g) \varepsilon(h)$ for all $a \in A$ and $g, h \in H$.
Then the associated crossed product is exactly the tensor product
of Hopf algebras $A \otimes H$.

$2)$ Let $A$, $H$ be two Hopf algebras and $f$ be the trivial
action. If $A$ is an $H$-module algebra via the coalgebra map
$\triangleright$ and the relation \equref{co3} is fulfilled then
the crossed product has the algebra structure given by the smash
product (\cite{Mo}) and it will be denoted by $A
\#^{\triangleright} H$.

$3)$ First recall that for any set $X$, the free module of basis
$X$, $kX$ can be made into a coalgebra by considering the
structure given by: $\Delta(x) = x \otimes x$ and $\varepsilon(x)
= 1$, for all $x\in X$. This will be the coalgebra structure we
consider on $kG$, $kH$ as well as on $k[H \#^{\triangleright}_{f}
\, G]$. In this setting, if $(H, G, f, \triangleright)$ is a
normalized crossed system of groups (see \cite{am2} for
definitions and further properties) then $(kH, kG,$
$\widetilde{f}, \widetilde{\triangleright})$ is a crossed system
of Hopf algebras where $\widetilde{f}$ and
$\widetilde{\triangleright}$ are obtained by linearizing the maps
$f$ and $\triangleright$. Furthermore, there exists an isomorphism
of Hopf algebras: $\varphi: k[H \times^{\triangleright}_{f} \, G]
\rightarrow kH \#^{\widetilde{\triangleright}}_{\widetilde{f}} \,
kG$ given by $\varphi\bigl((h, g)\bigl) = h \# g$ for all $h \in
H$, $g \in G$ and extended linearly to $k[H
\times^{\triangleright}_{f} \, G]$. Moreover, since coalgebra maps
between two group algebras are induced by maps between the
corresponding groups of the group algebras then any crossed
product of Hopf algebras between group algebras arises as above.

$4)$ Let $A$ and $H$ be two Hopf algebras such that $H$ is
cocommutative and $\gamma: H \rightarrow A$ a unitary coalgebra
map. Define:
$$
\triangleright:= \triangleright_{\gamma}: H \otimes A \rightarrow A ,
\qquad h\triangleright a := \gamma(h_{(1)})a\gamma^{-1}(h_{(2)})
$$
$$
f := f_{\gamma}: H \otimes H \rightarrow A, \qquad f(h,\,g)
= \gamma(h_{(1)}) \gamma(g_{(1)}) \gamma^{-1}(h_{(2)}g_{(2)})
$$
Then, using the cocommutativity of $H$, is a routine computation
to check that $(A, H, \triangleright_{\gamma}, f_{\gamma})$ is a
crossed system of Hopf algebras and, moreover, the map given by:
$$
\varphi : A \#^{\triangleright}_{f} \, H \to A \ot H, \quad
\varphi (a \# h ) = a \gamma (h_{(1)}) \ot h_{(2)}
$$
for all $a\in A$ and $h\in H$ is an isomorphism of Hopf algebras.
\end{examples}

\section{Basic properties of crossed products}

In this section we collect some fundamental properties of crossed
products of Hopf algebras. For a crossed product $A
\#^{\triangleright}_{f} \, H$ we can define the following
$k$-linear maps:
$$
i_A : A \to A \#^{\triangleright}_{f} \, H, \quad i_A (a) = a  \#
1_H, \qquad \pi_A : A \#^{\triangleright}_{f} \, H \to A, \quad
\pi_A (a \# h) = \varepsilon_H (h) a
$$
$$
i_H : H \to A \#^{\triangleright}_{f} \, H, \quad i_H (h) = 1_A \#
h, \qquad \pi_H : A \#^{\triangleright}_{f} \, H \to H, \quad
\pi_H (a \# h) = \varepsilon_A (a) h
$$
for all $a \in A$, $h\in H$. It is straightforward to see that
$i_{H}$ and $\pi_{A}$ are coalgebra maps while $i_{A}$ and
$\pi_{H}$ are Hopf algebra maps. The crossed product $A
\#^{\triangleright}_{f} \, H$ fits into the following exact
sequence:
\begin{eqnarray*}
\xymatrix{ 0 \ar[r] & A \ar[r]^{i_{A}} & {A
\#^{\triangleright}_{f} \, H} \ar[r]^{\pi_{H}} & H \ar[r] & 0 }
\end{eqnarray*}
such that $\pi_{H}: A \#^{\triangleright}_{f} \, H \rightarrow H$
admits a section which is a coalgebra map and $i_{A}$ admits a
retraction which is also a coalgebra map.

Recall that a Hopf subalgebra $A$ of a Hopf algebra $H$ is called
\textit{normal} if $x_{(1)}aS(x_{(2)}) \in A$ and
$S(x_{1})ax_{(2)} \in A$ for all $x \in H$, $a \in A$. We say that
a Hopf algebra $E$ factorizes through a Hopf subalgebra $A$ and a
subcoalgebra $H$ if the multiplication map $u: A\otimes H
\rightarrow E$, $u(a \otimes h) = ah $, for all $a\in A$, $h\in H$
is bijective.

The next theorem describes the main property that characterizes
the crossed product of Hopf algebras.
\begin{theorem}\thlabel{fact}
A Hopf algebra $E$ is isomorphic as a Hopf algebra to a crossed
product $A \#_{f}^{\triangleright} \, H$ of Hopf algebras if and
only if $E$ factorizes through a normal Hopf subalgebra $A$ and a
subcoalgebra $H$ such that $1_{E} \in H$.
\end{theorem}

\begin{proof}
Suppose first that $E \cong A \#_{f}^{\triangleright} \, H$. Then
$A \#_{f}^{\triangleright} \, H$ factorizes through $A$ and $H$ as
the map $u: A \otimes H \rightarrow A \#_{f}^{\triangleright} \,
H$ defined by $u(a \otimes h) = i_A (a) i_H(h) = a \# h$ is of
course bijective. It remains to prove that $A$ is a normal Hopf
subalgebra of $A \#_{f}^{\triangleright} \, H$. For any $b \in A$
and $a \# h \in A \#_{f}^{\triangleright} \, H$ we have:

\begin{eqnarray*}
&&S(a_{(1)} \# h_{(1)})(b \# 1) (a_{(2)} \# h_{(2)}) =\\
&{=}&\Bigl[S_{A}\bigl(f(S_{H}(h_{(2)}), h_{(3)})\bigl) \#
S_{H}(h_{(1)})\Bigl]\bigl(S_{A}(a_{(1)}) \# 1\bigl)(b \# 1)\\
&{=}& S_{A}\Bigl(f\bigl(S_{H}(h_{(5)},
h_{(6)})\bigl)\Bigl)\Bigl(S_{H}(h_{(4)}) \triangleright \bigl(
S_{A}(a_{(1)})b\bigl)\Bigl) \bigl(S_{H}(h_{(3)}) \triangleright
a_{(2)}\bigl) f \bigl(S_{H}(h_{(2)}), h_{(7)}\bigl)\\
&& \# S_{H}(h_{(1)})h_{(8)}\\
&{=}& S_{A}\Bigl(f\bigl(S_{H}(h_{(5)},
h_{(6)})\bigl)\Bigl)\Bigl(S_{H}(h_{(4)}) \triangleright \bigl(
S_{A}(a_{(1)})b\bigl)\Bigl) \bigl(S_{H}(h_{(3)}) \triangleright
a_{(2)}\bigl) \underline{f \bigl(S_{H}(h_{(1)})_{(1)}, h_{(7)(1)}\bigl)}\\
&& \# \underline{S_{H}(h_{(1)})_{(2)}h_{(7)(2)}}\\
&\stackrel{\equref{co3}}{=}& S_{A}\Bigl(f\bigl(S_{H}(h_{(5)},
h_{(6)})\bigl)\Bigl)\Bigl(S_{H}(h_{(4)}) \triangleright \bigl(
S_{A}(a_{(1)})b\bigl)\Bigl) \bigl(S_{H}(h_{(3)}) \triangleright
a_{(2)}\bigl) f \bigl(S_{H}(h_{(1)}), h_{(8)}\bigl)\\
&& \# S_{H}(h_{(2)})h_{(7)}\\
&{=}& S_{A}\Bigl(f\bigl(S_{H}(h_{(4)},
h_{(5)})\bigl)\Bigl)\underline{\Bigl(S_{H}(h_{(3)})_{(1)}
\triangleright \bigl( S_{A}(a_{(1)})b\bigl)\Bigl)
\bigl(S_{H}(h_{(3)})_{(2)} \triangleright
a_{(2)}\bigl)} f \bigl(S_{H}(h_{(1)}), h_{(7)}\bigl)\\
&& \# S_{H}(h_{(2)})h_{(6)}\\
&\stackrel{\equref{2}}{=}& S_{A}\Bigl(f\bigl(S_{H}(h_{(4)},
h_{(5)})\bigl)\Bigl)\underline{\Bigl(S_{H}(h_{(3)}) \triangleright
\bigl( S_{A}(a_{(1)})ba_{(2)}\bigl)\Bigl)}f\bigl(S_{H}(h_{(1)}),
h_{(7)}\bigl)\# \underline{S_{H}(h_{(2)})}h_{(6)}
\end{eqnarray*}
\begin{eqnarray*}
&\stackrel{\equref{co1}}{=}&
\underline{S_{A}\Bigl(f\bigl(S_{H}(h_{(3)})_{(1)},
h_{(5)(1)})\bigl)\Bigl)}\Bigl(S_{H}(h_{(2)}) \triangleright \bigl(
S_{A}(a_{(1)})ba_{(2)}\bigl)\Bigl)f\bigl(S_{H}(h_{(1)}),
h_{(7)}\bigl)\\
&& \# \underline{S_{H}(h_{(3)})_{(2)}h_{(5)(2)}}\\
&\stackrel{\equref{co3}}{=}& S_{A}\Bigl(f\bigl(S_{H}(h_{(3)}),
h_{(6)})\bigl)\Bigl)\Bigl(S_{H}(h_{(2)}) \triangleright \bigl(
S_{A}(a_{(1)})ba_{(2)}\bigl)\Bigl)f\bigl(S_{H}(h_{(1)}),
h_{(7)}\bigl)\# S_{H}(h_{(4)})h_{(5)}\\
&{=}& S_{A}\Bigl(f\bigl(S_{H}(h_{(3)}),
h_{(4)})\bigl)\Bigl)\Bigl(S_{H}(h_{(2)}) \triangleright \bigl(
S_{A}(a_{(1)})ba_{(2)}\bigl)\Bigl)f\bigl(S_{H}(h_{(1)}),
h_{(5)}\bigl)\# 1\\
\end{eqnarray*}
and the last line is obviously in $A$. We also have:
\begin{eqnarray*}
&&(a_{(1)} \# h_{(1)})(b \# 1)S(a_{(2)} \# h_{(2)}) =
\bigl(a_{(1)}(h_{(1)} \triangleright b) \# h_{(2)}\bigl)S(a_{(2)}
\# h_{(3)})\\
&{=}& \bigl(a_{(1)}(h_{(1)} \triangleright b) \#
h_{(2)}\bigl)\Bigl(S_{A}[f\bigl(S_{H}(g_{(4)}), \, g_{(5)}\bigl)]
\# S_{H}(g_{(3)})\Bigl) \bigl(S_{A}(a_{(2)}) \# 1_{H}\bigl)\\
&{=}& \Bigl(a_{(1)}(h_{(1)} \triangleright b)\bigl(h_{(2)}
\triangleright S_{A}[f\bigl(S_{H}(g_{(7)}), \,
g_{(8)}\bigl)]\bigl) f\bigl(h_{(3)}, S_{H}(h_{(6)})\bigl) \#
h_{(4)} S_{H}(h_{(5)})\Bigl)\\
&& \bigl(S_{A}(a_{(2)}) \# 1_{H}\bigl)\\
&{=}& \Bigl(a_{(1)}(h_{(1)} \triangleright b)\bigl(h_{(2)}
\triangleright S_{A}[f\bigl(S_{H}(g_{(5)}), \,
g_{(6)}\bigl)]\bigl) f\bigl(h_{(3)}, S_{H}(h_{(4)})\bigl) \# 1
\Bigl)\bigl(S_{A}(a_{(2)}) \# 1_{H}\bigl)
\end{eqnarray*}
the last line is again in $A$ so it follows that $A$ is a normal
Hopf subalgebra in $A \#_{f}^{\triangleright} \, H$ as desired.

Conversely, assume now that a Hopf algebra $E$ factorizes through
a normal Hopf subalgebra $A$ and a subcoalgebra $H$ such that
$1_{E} \in H$. It follows from \cite[Theorem 2.7]{am1} that $E$ is
isomorphic to a unified product with the multiplication, the
cocycle and the actions given by the formulas:
\begin{eqnarray*}
\cdot : H \otimes H \rightarrow h , \qquad \cdot &:=&
(\varepsilon_{A} \otimes Id) \circ \nu\\
f: H \otimes H \rightarrow A , \qquad f &:=& (Id \otimes
\varepsilon_{H}) \circ \nu\\
\triangleright: H \otimes A \rightarrow A , \qquad \triangleright
&:=& (Id \otimes \varepsilon_{H}) \circ \mu \\
\triangleleft: H \otimes A \rightarrow H , \qquad \triangleleft
&:=& (\varepsilon_{A} \otimes Id) \circ \mu
\end{eqnarray*}
where the coalgebra maps $\mu : H\otimes A \to A\ot H$ and $\nu :
H\ot H \to A\ot H$ are defined by:
$$
\mu (h\ot a) := u^{-1} (ha), \qquad  \nu (h\ot g) := u^{-1} (hg)
$$
for all $a \in A$ and $h$, $g \in H$. Since $A$ is a normal Hopf
subalgebra of $E$ we get:
$$\mu : H\otimes A \to A\ot H, \quad \mu (h\ot a) := u^{-1} (ha) =
h_{(1)}aS(h_{(2)}) \otimes h_{(3)}.$$ It follows from here that:
$h \triangleleft a = \varepsilon_{A}(a)h$, thus $\triangleleft$ is
trivial and so the unified product is a crossed product of Hopf
algebras (\cite[Examples 2.5 (2)]{am1}).

\end{proof}

Remark however that in the foregoing result the subcoalgebra $H$
is in fact a Hopf algebra but not a Hopf subalgebra of $E$. If a
Hopf algebra $E$ factorizes through two Hopf subalgebras $A$ and
$H$, with $A$ normal in $E$, then $E$ is isomorphic as a Hopf
algebra to a smash product as proved in \cite[Proposition
2.5]{BS}.

A classification result for crossed products of Hopf algebras can
be derived from \cite[Theorem 3.4]{am1} where a more general
approach to the classification problem in the context of unified
products is given. This result is the counterpart at the level of
Hopf algebras of the classical Schreier theorem on group
extensions. We recall from \cite{am1} the notion of coalgebra lazy
$1$-cocycle: given $A$ and $H$ two Hopf algebras, a coalgebra map
$u: H \rightarrow A$ is called a coalgebra lazy $1$-cocyle if
$u(1_H) = 1_A$ and the following compatibility holds:
\begin{equation}\eqlabel{lazy}
h_{(1)} \otimes
u(h_{(2)}) = h_{(2)} \otimes u(h_{(1)})
\end{equation}
for all $h\in H$.

\begin{proposition}
Let $A$, $H$ be two Hopf algebras and $A \#^{\triangleright}_{f}
\, H$, $A \#^{\triangleright'}_{f'} \, H$ be two crossed products
of Hopf algebras. The following are equivalent:\\
$(1)$ $A \#^{\triangleright}_{f} \, H \approx A
\#^{\triangleright'}_{f'} \, H$ (isomorphism of
Hopf algebras, left $A$-modules and right $H$-comodules);\\
$(2)$ There exists a coalgebra lazy $1$-cocyle $u: H \rightarrow
A$ such that:
\begin{enumerate}
\item[(1)] $h \triangleright' a = u^{-1}(h_{(1)})(h_{(2)}
\triangleright a)u(h_{(3)})$\\
\item[(2)] $f'(h,\,k) = u^{-1}(h_{(1)}) \bigl(h_{(2)}
\triangleright u^{-1}(k_{(1)}) f(h_{(3)},\,k_{(2)})\bigl)
u(h_{(4)}k_{(3)})$
\end{enumerate}
for all $a\in A$ and $h$, $g \in H$.

In this case we shall say that the crossed systems $(A, H,
\triangleright, f)$ and $(A, H, \triangleright', f')$ are
\emph{cohomologous} and we denote the equivalence classes of
crossed systems modulo such transformations by ${\mathcal H}^2 (H,
A)$.
\end{proposition}

\begin{proof}
It follows from the classification of unified products by
considering $\triangleleft$ to be the trivial action
(\cite[Theorem 3.4]{am1}).
\end{proof}

\begin{remark}
A crossed system of Hopf algebras which is equivalent to the
trivial one is called a \emph{coboundary}. An example of a
coboundary is given in \exref{1.2}, 4).
\end{remark}

The crossed product satisfies the following two universal
properties.

\begin{proposition}
Let $A$ and $H$ be two Hopf algebras and $A
\#^{\triangleright}_{f} \, H$ a crossed product of Hopf algebras.
Then, we have:
\begin{enumerate}
\item[(1)] For any Hopf algebra $X$, any Hopf algebra map $u: A
\rightarrow X$ and any coalgebra map $v: H \rightarrow X$ such
that the following compatibilities hold for all $h, g \in H$, $b
\in A$:
\begin{eqnarray}
\eqlabel{u1} u\bigl(f(h_{(1)},\,g_{(1))}\bigl) v(h_{(2)}g_{(2)})
&{=}& v(h)v(g)\\
\eqlabel{u2} u(h_{(1)} \triangleright b) v(h_{(2)}) &{=}& v(h)u(b)
\end{eqnarray}
there exists a unique Hopf algebra map $w: A
\#^{\triangleright}_{f} \, H \rightarrow X$ such that the
following diagram commutes:
\begin{equation}\eqlabel{d1}
\xymatrix { &{} &{A \#^{\triangleright}_{f} \, H}\ar@{..>}[d]^w & {}\\
& {A}\ar[ru]^{i_{A}}\ar[r]_{u} & {X} &
{H}\ar[lu]_{i_{H}}\ar[l]^{v}}
\end{equation}
\item[(2)] For any Hopf algebra $X$, any Hopf algebra map $v: X
\rightarrow H$ and any coalgebra map $u: X \rightarrow A$ such
that the following compatibilities hold for all $x, y \in X$:
\begin{eqnarray}
\eqlabel{u3} u(x_{(1)}) \otimes v(x_{(2)}) &{=}& u(x_{(2)})
\otimes v(x_{(1)})\\
\eqlabel{u4} u(xy) &{=}& u(x_{(1)})\bigl[v(x_{(2)}) \triangleright
u(y_{(1)})\bigl] f\bigl(v(x_{(3)}),\,v(y_{(2)})\bigl)
\end{eqnarray}
there exists a unique Hopf algebra map $w: X \rightarrow A
\#^{\triangleright}_{f} \, H$ such that the following diagram
commutes:
\begin{equation}\eqlabel{d2}
\xymatrix { &{} &{X}\ar@{..>}[d]^{w}\ar[ld]_{u}\ar[rd]^{v} & {}\\
& {A} & {A \#^{\triangleright}_{f} \,
H}\ar[r]_{\pi_{H}}\ar[l]^{\pi_{A}} & {H} }
\end{equation}
\end{enumerate}
\end{proposition}
\begin{proof}
$(1)$ Suppose first that such a map $w$ exists. Then, we have:
\begin{eqnarray*}
w (a \# g)  &{=}& w ((a \# 1)(1 \# g))= w (a \# 1)
w(1 \# g)\\
&{=}& (w \circ i_{A}) (a) (w \circ i_{H})(h)
\stackrel{\equref{d1}} = u (a) v(g)
\end{eqnarray*}
for all $a \# h \in A \#^{\triangleright}_{f} \, H$ and this
proves the uniqueness of $w$. Next, we prove that $w$ given by $w
(h \# g) = u (a) v(g)$ is a Hopf algebra map such that diagram
\equref{d1} commutes.
\begin{eqnarray*}
w\bigl((a \# h)(b \# g)\bigl) &{=}& w\bigl(a(h_{(1)}
\triangleright b)f(h_{(2)},\,g_{(1)}) \# h_{(3)}g_{(2)}\bigl)\\
&{=}& u(a)u(h_{(1)} \triangleright b)
\underline{u\bigl(f(h_{(2)},\,g_{(1)})\bigl)v(h_{(3)}g_{(2)})}\\
&\stackrel{\equref{u1}}{=}& u(a)\underline{u(h_{(1)}
\triangleright b)v(h_{(2)})}v(g)\\
&\stackrel{\equref{u2}}{=}& u(a)v(h)u(b)v(g) = w(a \# h)w(b \# g)
\end{eqnarray*}
Thus $w$ is an algebra map. The rest of  the proof is
straightforward and is left to the reader.\\
$(2)$ Define $w: X \rightarrow A \#^{\triangleright}_{f} \, H$ by
$w(x) = u(x_{(1)}) \otimes v(x_{(2)})$. The uniqueness of $w$ as
well as the relations $\pi_{H} \circ w = v$ and $\pi_{A} \circ w =
u$ are easy to establish. In what follows we prove that $w$ is an
algebra map.
\begin{eqnarray*}
w(x)w(y) &{=}& \bigl(u(x_{(1)}) \otimes
v(x_{(2)})\bigl)\bigl(u(y_{(1)}) \otimes v(y_{(2)})\bigl)\\
&{=}& u(x_{(1)})\bigl(v(x_{(2)})_{(1)} \triangleright
u(y_{(1)})\bigl) f\bigl(v(x_{(2)})_{(2)},\,v(y_{(2)})_{(1)}\bigl)
\# v(x_{(2)})_{(3)}v(y_{(2)})_{(2)}\\
&{=}& u(x_{(1)})\bigl(v(x_{(2)}) \triangleright u(y_{(1)})\bigl)
f\bigl(v(x_{(3)}),\,v(y_{(2)})\bigl) \# v(x_{(4)})v(y_{(3)})\\
&{=}& \underline{u(x_{(1)(1)})\bigl(v(x_{(1)(2)}) \triangleright
u(y_{(1)(1)})\bigl) f\bigl(v(x_{(1)(3)}),\,v(y_{(1)(2)})\bigl)} \#
v(x_{(2)})v(y_{(2)})\\
&\stackrel{\equref{u4}}{=}& u(x_{(1)}y_{(1)}) \#
v(x_{(2)})v(y_{(2)})\\
&{=}& u(x_{(1)}y_{(1)}) \# v(x_{(2)}y_{(2)}) = w(xy)
\end{eqnarray*}
\end{proof}

\section{Commutativity, semisimplicity, integrals for the crossed product}

In this section we study further properties of the crossed product
of Hopf algebras. In particular, we are interested in which of the
properties of the Hopf algebras $A$ and $H$ are preserved by the
crossed product $A \#_{f}^{\triangleright} \, H$. First, we should
note that given the coalgebra structure of the crossed product
(i.e. the tensor product coalgebra) it is straightforward to see
that $A \#_{f}^{\triangleright} \, H$ is cocommutative
(cosemisimple) if and only if both $A$ and $H$ are cocommutative
(cosemisimple).

Our next proposition extends the results obtained in \cite{os} on
the commutativity of crossed products of groups acting on rings.
The cocycle $f: H\ot H \to A$ is called \textit{symmetric} if
$f(g,h) = f(h,g)$ for all $g, h \in H$.

\begin{proposition}\prlabel{3.1}
Let $A \#_{f}^{\triangleright} \, H$ be a crossed product of Hopf
algebras. Then $A \#_{f}^{\triangleright} \, H$ is commutative if
and only if $A$ and $H$ are commutative, $\triangleright$ is
trivial and $f$ is symmetric.
\end{proposition}
\begin{proof}
Suppose first that $A \#_{f}^{\triangleright} \, H$ is
commutative. Then, we have:
\begin{equation}\eqlabel{com}
a(h_{(1)} \triangleright c)f\bigl(h_{(2)},g_{(1)}\bigl) \#
h_{(3)}g_{(2)} = c(g_{(1)} \triangleright a) f\bigl(g_{(2)},
h_{(1)}\bigl) \# g_{(3)}h_{(2)}
\end{equation}
for all $a, c \in A$ and $h, g \in H$.\\
By taking $a=c=1$ and then applying $\varepsilon \otimes Id$ in
\equref{com} we get $hg=gh$ for all $h, g \in H$ that is, $H$ is
commutative. Moreover, if we apply $Id \otimes \varepsilon$ to the
same equation \equref{com} with $a=c=1$ we get that $f(h, \,
g) = f(g, \ h)$ for all $h, g \in H$ that is, $f$ is symmetric.\\
Next, if we take $c=1$ and $h=1$ in \equref{com} and then we apply
$Id \otimes \varepsilon$ we obtain $g \triangleright a =
\varepsilon(g)a$ for all $a \in A$, $g \in H$, that is,
$\triangleright$ is trivial. Finally, $A$ is commutative as a
subalgebra in the commutative Hopf algebra $A
\#_{f}^{\triangleright} \, H$.

Assume now that $A$ and $H$ are commutative Hopf algebras,
$\triangleright$ is trivial and $f$ is symmetric. We then have:
\begin{eqnarray*}
(a \# h)(c \# g) &{=}& a c f\bigl(h_{(1)}\, \, g_{(1)}\bigl) \#
h_{(2)}g_{(2)}\\
&{=}& c a f\bigl(g_{(1)}, \, h_{(1)}\bigl) \# g_{(2)}h_{(2)}\\
&{=}& (c \# g)(a \# h)
\end{eqnarray*}
for all $a \# h$, $c \# g \in A \#_{f}^{\triangleright} \, H$.
Thus $A \#_{f}^{\triangleright} \, H$ is commutative.
\end{proof}

\begin{proposition}\prlabel{inv}
Let $A \#_{f}^{\triangleright} \, H$ be a crossed product of Hopf
algebras.
\begin{enumerate}
\item[(1)] If $A \#_{f}^{\triangleright} \, H$ is
involutory then both $A$ and $H$ are involutory Hopf algebras.\\
\item[(2)] Suppose that $H$ is cocommutative. Then $A
\#_{f}^{\triangleright} \, H$ is involutory if and only if $A$ is
involutory and $g_{(1)} \triangleright
f\bigl(S_{H}(g_{(2)}),\,g_{(3)}\bigl) =
f\bigl(g_{(1)},\,S_{H}(g_{(2)})\bigl)$ for all $g \in H$.
\end{enumerate}
\end{proposition}
\begin{proof}
$(1)$ Since $A \#_{f}^{\triangleright} \, H$ is involutory and $A$
is a Hopf subalgebra of $A \#_{f}^{\triangleright} \, H$ it
follows that $A$ is indeed involutory. Furthermore, we also have:
\begin{eqnarray*}
&&1 \# g = S^{2}(1 \# g) \stackrel{\equref{antipod}}=
S\Bigl(S_{A}\bigl(f(S_{H}(g_{(2)}),\,g_{(3)})\bigl) \#
S_{H}(g_{(1)})\Bigl)\\
&\stackrel{\equref{antipod}}{=}&
\Bigl[S_{A}\Bigl(f\bigl(S_{H}^{2}(g_{(2)}),\,S_{H}(g_{(1)})\bigl)\Bigl)
\# S_{H}^{2}(g_{(3)})\Bigl]
\Bigl[\underline{S_{A}^{2}\bigl(f(S_{H}^{2}(g_{(4)}),\,g_{(5)})\bigl)}
\# 1\Bigl]\\
&{=}&
\Bigl[S_{A}\Bigl(f\bigl(S_{H}^{2}(g_{(2)}),\,S_{H}(g_{(1)})\bigl)\Bigl)
\# S_{H}^{2}(g_{(3)})\Bigl]
\Bigl(f\bigl(S_{H}^{2}(g_{(4)}),\,g_{(5)}\bigl)
\# 1\Bigl)\\
&{=}&
S_{A}\Bigl(f\bigl(S_{H}^{2}(g_{(2)}),\,S_{H}(g_{(1)})\bigl)\Bigl)\Bigl(S_{H}^{2}(g_{(3)})
\triangleright f\bigl(S_{H}(g_{(5)}),\,g_{(6)}\bigl)\Bigl) \#
S_{H}^{2}(g_{(4)})
\end{eqnarray*}
Hence, we obtained:
\begin{equation}\eqlabel{H}
1 \# g =
S_{A}\Bigl(f\bigl(S_{H}^{2}(g_{(2)}),\,S_{H}(g_{(1)})\bigl)\Bigl)\Bigl(S_{H}^{2}(g_{(3)})
\triangleright f\bigl(S_{H}(g_{(5)}),\,g_{(6)}\bigl)\Bigl) \#
S_{H}^{2}(g_{(4)})
\end{equation}
If we apply $\varepsilon \otimes Id$ in \equref{H} we get
$S_{H}^{2}(g) = g$ for all $g \in H$.\\
$(2)$ Suppose first that $A \#_{f}^{\triangleright} \, H$ is
involutory. It follows from $(1)$ that $A$ and $H$ are involutory.
Remark however that $H$ is always involutory by the assumption of
being cocommutative. Since $A \#_{f}^{\triangleright} \, H$ is
involutory it follows again by $(1)$ that relation \equref{H}
holds for all $h \in H$. By applying $Id \otimes \varepsilon$ in
\equref{H} and having in mind that $S_{H}^{2}(g) = g$ for all $g
\in H$, we get:
\begin{equation}\eqlabel{19}
S_{A}\Bigl(f\bigl(g_{(2)},\,S_{H}(g_{(1)})\bigl)\Bigl)\Bigl(g_{(3)}
\triangleright f\bigl(S_{H}(g_{(4)}),\,g_{(5)}\bigl)\Bigl) =
\varepsilon(g)
\end{equation}
As $H$ is cocommutative, \equref{19} is equivalent to:
\begin{equation}\eqlabel{20}
S_{A}\Bigl(f\bigl(g_{(4)},\,S_{H}(g_{(1)})\bigl)\Bigl)\Bigl(g_{(2)}
\triangleright f\bigl(S_{H}(g_{(3)}),\,g_{(5)}\bigl)\Bigl) =
\varepsilon(g)
\end{equation}
Since $f$ is a coalgebra map it follows, after inverting $f$ in
\equref{20}, that we have: $$g_{(1)} \triangleright
f\bigl(S_{H}(g_{(2)}),\,g_{(3)}\bigl) =
f\bigl(g_{(1)},\,S_{H}(g_{(2)})\bigl)$$ for all $g \in H$.\\
Assume now that $A$ is involutory and
\begin{equation}\eqlabel{inv}
g_{(1)} \triangleright f\bigl(S_{H}(g_{(2)}),\,g_{(3)}\bigl) =
f\bigl(g_{(1)},\,S_{H}(g_{(2)})\bigl)
\end{equation}
holds for all $g \in H$. As mentioned before, $H$ is also
involutory from the cocommutativity assumption. Since $T = \{a \#
1_{H} ~|~ a \in A\} \cup \{1_{A} \# h ~|~ h \in H\}$ generates $A
\#_{f}^{\triangleright} \, H$ as an algebra and $S^{2}(a \# h) =
S^{2}(a \# 1)S^{2}(1 \# h)$, we only need to prove that $S^{2}(a
\# 1) = a \# 1$ and $S^{2}(1 \# h) = 1 \# h$ for all $a \in A$, $h
\in H$. Since $ S^{2}(a \# 1) = S^{2}_{A}(a) \# 1$ it follows
trivially that $S^{2}(a \# 1) = a \# 1$ by the involutivity of
$A$. As we already noticed in $(1)$, we have:
\begin{eqnarray*}
S^{2}(1 \# g) &{=}&
S_{A}\Bigl(f\bigl(g_{(2)},\,S_{H}(g_{(1)})\bigl)\Bigl)\underline{\Bigl(g_{(3)}
\triangleright f\bigl(S_{H}(g_{(5)}),\,g_{(6)}\bigl)\Bigl) \#
g_{(4)}}\\
&\stackrel{\equref{co1}}{=}&
S_{A}\Bigl(f\bigl(g_{(2)},\,S_{H}(g_{(1)})\bigl)\Bigl)\underline{\Bigl(g_{(4)}
\triangleright f\bigl(S_{H}(g_{(5)}),\,g_{(6)}\bigl)\Bigl)} \#
g_{(3)}\\
&\stackrel{\equref{inv}}{=}&S_{A}\Bigl(f\bigl(g_{(2)},\,S_{H}(g_{(1)})\bigl)\Bigl)f\bigl(g_{(4)},\,S_{H}(g_{(5)})\bigl)
\# g_{(3)}\\
&{=}&S_{A}\Bigl(f\bigl(g_{(3)},\,S_{H}(g_{(2)})\bigl)\Bigl)f\bigl(g_{(4)},\,S_{H}(g_{(1)})\bigl)
\# g_{(5)}\\
&{=}& 1 \# g
\end{eqnarray*}
for all $g \in H$. Thus $A \#_{f}^{\triangleright} \, H$ is
involutory and the proof is finished.
\end{proof}
\begin{remark}
A more general result concerning the involutivity of the crossed
product $A \#_{f}^{\triangleright} \, H$ can be obtained by
dropping the cocommutativity assumption on $H$. However the result
is less transparent.
\end{remark}

Our next result indicates a way of constructing integrals for the
crossed products and proves that the crossed product of two Hopf
algebras is semisimple if and only if both Hopf algebras are
semisimple. This result is true regardless of the characteristic
of the field $k$ (in characteristic $0$ the result can be easily
derived from the well-known result in \cite{LR}).

\begin{proposition}\prlabel{3.3}
Let $A \#_{f}^{\triangleright} \, H$ be a crossed product of
finite dimensional Hopf algebras.
\begin{enumerate}
\item[(1)] If $x_{A}$ is a right integral in A and $x_{H}$ a right
integral in $H$ then $x_{A} \# x_{H}$ is a right
integral in $A \#_{f}^{\triangleright} \, H$;\\
\item[(2)] If $\Sigma t^{1} \otimes t^{2}$ is a right integral in
$A \#_{f}^{\triangleright} \, H$ then $z_{H} := \Sigma
\varepsilon(t^{1})t^{2}$ is a right integral in $H$. If $\Sigma
t^{1} \otimes t^{2}$ is a left integral in $A
\#_{f}^{\triangleright} \, H$ then $z_{A} := \Sigma
t^{1}\varepsilon(t^{2})$ is a left integral in $A$;\\
\item[(3)] $A \#_{f}^{\triangleright} \, H$ is semisimple if and
only if $A$ and $H$ are both semisimple.
\end{enumerate}
\end{proposition}

\begin{proof}
$(1)$ Let $a \# h \in A \#_{f}^{\triangleright} \, H$. We then
have:
\begin{eqnarray*}
(x_{A} \# x_{H})(a \# h) &{=}& \underline{x_{A}
\bigl((x_{H})_{(1)} \triangleright a \bigl)} f\bigl((x_{H})_{(2)},
\, h_{(1)}\bigl)\#
(x_{H})_{(3)}h_{(2)}\\
&{=}& x_{A} \varepsilon \bigl((x_{H})_{(1)} \triangleright a
\bigl)f\bigl((x_{H})_{(2)}, \, h_{(1)}\bigl)\#
(x_{H})_{(3)}h_{(2)}\\
&{=}& \varepsilon(a) \underline{x_{A} f\bigl((x_{H})_{(1)}, \,
h_{(1)}\bigl)}\# (x_{H})_{(2)}h_{(2)}\\
&{=}& \varepsilon(a) x_{A} \varepsilon \Bigl
(f\bigl((x_{H})_{(1)}, \,
h_{(1)}\bigl)\Bigl)\# (x_{H})_{(2)}h_{(2)}\\
&{=}& \varepsilon(a) x_{A} \# \underline{x_{H}h}\\
&{=}& \varepsilon(a) \varepsilon(h) x_{A} \# x_{H}\\
&{=}& \varepsilon(a \# h) x_{A} \# x_{H}
\end{eqnarray*}
where we used the fact that $\triangleright$ and $f$ are coalgebra
maps and $x_{A}$, $x_{H}$ are right integrals in $A$ respectively
$H$.\\
$(2)$ As $\Sigma t^{1} \# t^{2}$ is a right integral in $A
\#_{f}^{\triangleright} \, H$ we have:
\begin{equation}\eqlabel{int}
\Sigma t^{1}(t^{2}_{(1)} \triangleright a)f(t^{2}_{(2)}, \,
h_{(1)}) \# t^{2}_{(3)}h_{(2)} = \Sigma \varepsilon(a)
\varepsilon(h) t^{1} \# t^{2}
\end{equation}
for all $a \# h \in A \#_{f}^{\triangleright} \, H$. By applying
$\varepsilon \otimes Id$ in \equref{int} and considering $a \# h
:= 1_{A} \# g$ we get, using the fact that $f$ is a coalgebra map,
that: $\Sigma \varepsilon(t^{1})t^{2}g = \Sigma \varepsilon(g)
\varepsilon(t^{1})t^{2}$ for every $g \in H$. But the last
equality is equivalent to: $z_{H}g = \varepsilon(g)z_{H}$. Hence
$z_{H}$ is a right integral in $H$. In the same manner it can be
proved that if $\Sigma t^{1} \otimes t^{2}$ is a left integral in
$A \#_{f}^{\triangleright} \, H$ then $z_{A} := \Sigma
t^{1}\varepsilon(t^{2})$ is a left integral in $A$.

$(3)$ Assume first that both $A$ and $H$ are semisimple algebras.
It follows from Maschke's theorem that there exist right integrals
$x_{A} \in A$ and $x_{H} \in H$ such that $\varepsilon(x_{A}) =
\varepsilon (x_{H}) = 1$. From $(1)$ we have that $x_{A} \# x_{H}$
is a right integral in $A \#_{f}^{\triangleright} \, H$. Moreover,
we have $\varepsilon(x_{A} \# x_{H}) =
\varepsilon(x_{A})\varepsilon (x_{H}) = 1$. Thus $A
\#_{f}^{\triangleright} \, H$ is semisimple.

Suppose now that $A \#_{f}^{\triangleright} \, H$ is semisimple.
Since $A$ is a Hopf subalgebra in $A \#_{f}^{\triangleright} \, H$
then $A$ is semisimple. Let $\Sigma t^{1} \otimes t^{2}$ be a
right integral in $A \#_{f}^{\triangleright} \, H$. Hence, by
$(2)$, $z_{H} := \Sigma \varepsilon(t^{1})t^{2} \in H$ is a right
integral in $H$. Since $\varepsilon(\Sigma t^{1} \# t^{2}) = 1$ it
follows that we also have $\varepsilon(z_{H}) = 1$. Hence $A$ and
$H$ are both semisimple algebras.
\end{proof}

\section{Braided structures on the crossed product of Hopf algebras}

In this section we describe the coquasitriangular or braided
structures on the crossed product of Hopf algebras. In other
words, we determine all braided structures that can be defined on
the monoidal category of $A \#_{f}^{\triangleright} \, H$ -
comodules. The notion of coquasitriangular bialgebra (Hopf
algebra) appeared for the first time in \cite{Mj} and it was
formulated and studied by Larson and Towber in \cite{LT}.

First we recall from \cite{DT2} that if $A$ and $H$ are two Hopf
algebras and $\alpha: A \otimes H \rightarrow k$ is a $k$-linear
map which fulfills the compatibilities:
\begin{enumerate}
\item[(BR1)] $p(xy , z) = p(x , z_{(1)}) p(y ,
z_{(2)})$\\
\item[(BR2)] $p(1 , x) = \varepsilon(x)$\\
\item[(BR3)] $p(x , yz) = p(x_{(1)} , z) p(x_{(2)}
, y)$\\
\item[(BR4)] $p(y , 1) = \varepsilon(y)$\\
\end{enumerate}
for all $x$, $y \in A$, $z \in H$, then $\alpha$ is called
\textit{skew pairing} on $(A, H)$

Moreover, recall from \cite{LT} that a Hopf algebra $H$ is called
\textit{braided} or \textit{coquasitriangular} if there exists a
linear map $p: H \otimes H \rightarrow k$ such that relations
$(BR1)-(BR4)$ are fulfilled and
\begin{enumerate}
\item[(BR5)] $p(x_{(1)} , y_{(1)})x_{(2)}y_{(2)} = y_{(1)}x_{(1)}
p(x_{(2)} , y_{(2)})$ \end{enumerate} holds for all $x$, $y$, $z
\in H$.

\begin{definition}
Let $A$, $H$ be two Hopf algebras, $f: H \otimes H \rightarrow A$
a coalgebra map and $p: A\otimes A \rightarrow k$ a braiding on
$A$. A linear map $u : A \otimes H \rightarrow k$ is called
\textit{(p,f) - right skew pairing on $(A, H)$} if the following
compatibilities are fulfilled for any $a$, $b \in A$, $g$, $t \in
H$:
\begin{enumerate}
\item[(RS1)] $u(ab , t) = u(a , t_{(1)}) u(b , t_{(2)})$\\
\item[(RS2)] $u(1 , h) = \varepsilon(h)$\\
\item[(RS3)] $u(a_{(1)} , g_{(2)}t_{(2)}) p\bigl(a_{(2)} ,
f(g_{(1)} , t_{(1)})\bigl) = u(a_{(1)} , t) u(a_{(2)} , g)$\\
\item[(RS4)] $u(a , 1) = \varepsilon(a)$
\end{enumerate}
\end{definition}

\begin{definition}
Let $A$, $H$ be two Hopf algebras, $f: H \otimes H \rightarrow A$
a coalgebra map and $p: A\otimes A \rightarrow k$ a braiding on
$A$. A linear map $v: H \otimes A \rightarrow k$ is called
\textit{(p,f) - left skew pairing on $(H, A)$} if the following
compatibilities are fulfilled for any $b$, $c \in A$, $h$, $g \in
H$:
\begin{enumerate}
\item[(LS1)] $p\bigl(f(h_{(1)} , g_{(1)}) , c_{(1)}\bigl)
v(h_{(2)}g_{(2)} , c_{(2)}) = v(h , c_{(1)}) v(g , c_{(2)})$\\
\item[(LS2)] $v(1 , a) = \varepsilon(a)$\\
\item[(LS3)] $v(h , bc) = v(h_{(1)} , c) v(h_{(2)} , b)$\\
\item[(LS4)] $v(h , 1) = \varepsilon(h)$
\end{enumerate}
\end{definition}

\begin{examples}
$1)$ If $f = \varepsilon_{H} \otimes \varepsilon_{H}$ then the
notions of (p,f) - right skew pairing and (p,f) - left skew
pairing coincide with the notion of skew pairing on $(A, H)$
respectively $(H, A)$.\\
$2)$ Let $L = < t ~|~ t^{n} = 1 >$ and $G = < g ~|~ g^{m} = 1>$ be
two cyclic groups of orders $n$ respectively $m$ and consider the
group Hopf algebras $A = k[L]$ and $H = k[G]$. In this setting a
coalgebra map $f: k[G] \otimes k[G] \rightarrow k[L]$ is
completely determined by a map $\alpha: \{0, 1, ...,m-1\} \times
\{0, 1, ..., m-1\} \rightarrow \{0, 1, ..., n-1\}$ such that
$f(g^{i},\ g^{j}) = t^{\alpha(i,\ j)}$. Recall that the braidings
on a group Hopf algebra $k[L]$ are in one-to-one correspondence
with the bicharacters on $L$. More precisely, the braidings on
$k[L]$ are given by: $$p: k[L] \otimes k[L] \rightarrow k, \qquad
p(t^{a},\ t^{b}) = \tau^{ab}$$ where $a, b \in \overline{0, n-1}$
and $\tau \in k$ such that $\tau^{n}=1$. Then there exists $u:
k[L] \otimes k[G] \rightarrow k$ a $(p,f)$ - right skew pairing if
and only if $\alpha$ is a symmetric map and there exists $\upsilon
\in k$ such that $\upsilon^{n} = 1$ and $\upsilon^{m} =
\tau^{\alpha(1,\ m-1)}$. In this case the $(p, f)$ - right skew
pairing $u: k[L] \otimes k[G] \rightarrow k$ is given by:
$$u(t^{a},\ g^{b}) = \upsilon^{ab}\tau^{-\alpha(1,\ b-1)}, \qquad
a \in \overline{0, n-1}, b \in \overline{0, m-1}.$$
\end{examples}

\begin{definition}
Let $A$ and $H$ be two Hopf algebras, $f: H \otimes H \rightarrow
A$ a coalgebra map, $p: A\otimes A \rightarrow k$ a braiding on
$A$, $u: A \otimes H \rightarrow k$ a $(p , f)$ - right skew
pairing on $(A, H)$ and $v: H \otimes A \rightarrow k$ a $(p , f)$
- left skew pairing on $(H, A)$. A linear map $\tau: H \otimes H
\rightarrow k$ is called \textit{$(u,v)$ - skew braiding on $H$}
if the following compatibilities are fulfilled for all $h$, $g$,
$t \in H$:
\begin{enumerate}
\item[(SBR1)] $u\bigl( f(h_{(1)} , g_{(1)}) , t_{(1)}\bigl)
\tau(h_{(2)}g_{(2)} , t_{(2)}) = \tau(h , t_{(1)}) \tau(g ,
t_{(2)})$\\
\item[(SBR2)] $\tau(1 , h) = \varepsilon(h)$\\
\item[(SBR3)] $\tau(h_{(1)} , g_{(2)}t_{(2)}) v\bigl(h_{(2)} ,
f(g_{(1)} , t_{(1)}) \bigl) = \tau(h_{(1)} , t) \tau(h_{(2)} ,
g)$\\
\item[(SBR4)] $\tau(g , 1) = \varepsilon(g)$\\
\item[(SBR5)] $\tau(h_{(1)} , g_{(1)})h_{(2)}g_{(2)} =
g_{(1)}h_{(1)} \tau(h_{(2)} , g_{(2)})$
\end{enumerate}
\end{definition}

Remark that if $f = \varepsilon_{H} \otimes \varepsilon_{H}$ then
the $(u,v)$ - skew braiding $\tau$ is a regular braiding on $H$.

\begin{theorem}\thlabel{4.5}
Let $(A, H, \triangleright, f)$ be a crossed system of Hopf
algebras. The following are equivalent:
\begin{enumerate}
\item[1)] $(A \#_{f}^{\triangleright} \, H, \sigma)$ is a braided
Hopf algebra\\
\item[2)] There exist four linear maps $p: A \otimes A \rightarrow
k$, $\tau: H \otimes H \rightarrow k$, $u: A \otimes H \rightarrow
k$, $v: H \otimes A \rightarrow k$ such that $(A, p)$ is a braided
Hopf algebra, $u$ is a $(p,f)$ - right skew pairing on $(A, H)$,
$v$ is a $(p,f)$ - left skew pairing on $(H, A)$, $\tau$ is a $(u,
v)$ - skew braiding on $H$ and the following compatibilities are
fulfilled:
\end{enumerate}
\begin{eqnarray}
\eqlabel{4.1} v(h_{(1)} , b_{(1)}) (h_{(2)} \triangleright
b_{(2)}) \otimes h_{(3)} &{=}& b_{(1)} \otimes h_{(1)}v(h_{(2)} ,
b_{(2)})\\
\eqlabel{4.2} (g_{(1)} \triangleright a_{(1)}) \otimes
g_{(2)}u(a_{(2)} , g_{(3)}) &{=}& u(a_{(1)} , g_{(1)})a_{(2)}
\otimes
g_{(2)}\\
\eqlabel{4.3} \tau(h_{(1)} , g_{(1)}) f(h_{(2)} , g_{(2)}) &{=}&
f(g_{(1)} , h_{(1)}) \tau(h_{(2)} , g_{(2)})\\
\eqlabel{4.4} u(a_{(1)} , g_{(2)}) p(a_{(2)} , g_{(1)}
\triangleright c) &{=}& p(a_{(1)} , c) u(a_{(2)} , g)\\
\eqlabel{4.5} \tau(h_{(1)} , g_{(2)}) v(h_{(2)} , g_{(1)}
\triangleright c) &{=}& v(h_{(1)} , c) \tau(h_{(2)} , g)\\
\eqlabel{4.6} p(h_{(1)} \triangleright b , c_{(1)}) v(h_{(2)} ,
c_{(2)}) &{=}& v(h , c_{(1)}) p(b , c_{(2)})\\
\eqlabel{4.7} u(h_{(1)} \triangleright b , t_{(1)}) \tau(h_{(2)} ,
t_{(2)}) &{=}& \tau(h , t_{(1)}) u(b , t_{(2)})
\end{eqnarray}
and the braiding $\sigma: (A \#_{f}^{\triangleright} \, H) \otimes
(A \#_{f}^{\triangleright} \, H) \rightarrow k$ is given by:
\begin{equation}\eqlabel{39}
\sigma(a \# h, b \# g) = u(a_{(1)} , g_{(1)}) p(a_{(2)} , b_{(1)})
\tau(h_{(1)} , g_{(2)}) v(h_{(2)} , b_{(2)})
\end{equation}
for all $a, b, c \in A$ and $h, g, t \in H$.
\end{theorem}
\begin{proof}
Suppose first that $(A \#_{f}^{\triangleright} \, H, \sigma)$ is a
braided Hopf algebra. We define the following linear maps:
\begin{eqnarray*}
p: A \otimes A \rightarrow k, \qquad p(a, b) &=& \sigma(a \otimes
1
, b \otimes 1)\\
\tau: H \otimes H \rightarrow k, \qquad \tau(h, g) &=& \sigma(1
\otimes h , 1 \otimes g)\\
u: A \otimes H \rightarrow k, \qquad u(a , h) &=& \sigma(a \otimes
1
, 1 \otimes h)\\
v: H \otimes A \rightarrow k, \qquad v(h, a) &=& \sigma(1 \otimes
h, a \otimes 1)
\end{eqnarray*}
Before going into the proof we collect here some compatibilities
satisfied by the maps defined above which will be useful in the
sequel. The following are just easy consequences of the fact that
$\sigma$ is a braiding on $A \#_{f}^{\triangleright} \, H$ and,
hence, it satisfies relations (BR2) and (BR4):
\begin{eqnarray}
\eqlabel{a} p(1 , b) = \varepsilon(b) = p(b , 1)\\
\eqlabel{b} \tau(1 , h) = \varepsilon(h) = \tau(h , 1)\\
\eqlabel{c} u(1 , h) = \varepsilon(h) , \quad u(a , 1) =
\varepsilon(a)\\
\eqlabel{d} v(1 , a) = \varepsilon(a) , \quad v(h , 1) =
\varepsilon(h)
\end{eqnarray}
Remark that from relation \equref{b} it follows that $\tau$
fulfills (SBR2) and (SBR4) while from relation \equref{d} we can
derive that $v$ fulfills (LS2) and (LS4).

It is easy to see that relation \equref{39} indeed holds:
\begin{eqnarray*}
&&\sigma(a \# h,\ b \# g) = \sigma \bigl((a \# 1)(1 \# h), (b \# 1)(1 \# g)\bigl) = \\
&\stackrel{(BR1)}{=}& \sigma \bigl(a \# 1, (b_{(1)} \# 1)(1 \#
g_{(1)})\bigl) \sigma \bigl((1 \# h), (b_{(2)} \# 1)(1 \#
g_{(2)})\bigl)\\
&\stackrel{(BR3)}{=}& \sigma (a_{(1)} \# 1 , 1 \# g_{(1)}) \sigma
(a_{(2)} \# 1 , b_{(1)} \# 1) \sigma (1 \# h_{(1)} , 1 \# g_{(2)})
\sigma \bigl(1 \# h_{(2)} , b_{(2)} \# 1)\\
&{=}&u(a_{(1)} , g_{(1)}) p(a_{(2)} , b_{(1)}) \tau(h_{(1)} ,
g_{(2)}) v(h_{(2)} , b_{(2)})
\end{eqnarray*}
Next we prove that $(A, p)$ is a braided Hopf algebra, $u$ is a
$(p,f)$ - right skew pairing on $(A, H)$, $v$ is a $(p,f)$ - left
skew pairing on $(H, A)$ and $\tau$ is a $(u, v)$ - skew braiding
on $H$. Having in mind that $(A \#_{f}^{\triangleright} \, H,
\sigma)$ is a braided Hopf algebra it is straightforward to see
that $(A,p)$ is a braided Hopf algebra by considering $x = a \#
1$, $y = b \# 1$ and $z = c \# 1$ in (BR1) $-$ (BR5). As $(A
\#_{f}^{\triangleright} \, H, \sigma)$ is braided we get from
(BR1):
\begin{eqnarray}
\eqlabel{(BR1)} \sigma \bigl(a(h_{(1)} \triangleright b) f(h_{(2)}
, g_{(1)}) \# h_{(3)}g_{(2)} , c \# t \bigl) = \sigma ( a \otimes
h, c_{(1)} \otimes t_{(1)}) \sigma(b \otimes g, c_{(2)} \otimes
t_{(2)})
\end{eqnarray}
By considering $h = g = 1$ and $c = 1$ we get relation (RS1).
Moreover, by (BR3) we have:
\begin{eqnarray}
\eqlabel{(BR3)} \sigma \bigl(a \# h , b(g_{(1)} \triangleright c)
f(g_{(2)} , t_{(1)}) \# g_{(3)}t_{(2)} \bigl)  = \sigma ( a_{(1)}
\otimes h_{(1)} , c \otimes t) \sigma(a_{(2)} \otimes h_{(2)}, b
\otimes g)
\end{eqnarray}
Considering $b = c = 1$ and $h = 1$ in the above relation and
using the decomposition of $\sigma$ from \equref{39} we get:
$$
u(a_{(1)} , g_{(3)}t_{(3)}) p\bigl(a_{(2)} , f(g_{(1)} ,
t_{(1)})\bigl) \tau(1 , g_{(4)}t_{(4)}) v\bigl(1 , f(g_{(2)} ,
t_{(2)})\bigl) = u(a_{(1)} , t) u(a_{(2)} , g)
$$
Now using relations \equref{b} and \equref{d} we get (RS3). Hence
we proved that $u$ is a $(p,f)$ - right skew pairing on $(A, H)$.
Using \equref{(BR1)} again for $a = b = 1$ and $t = 1$ we get:
\begin{eqnarray*}
u \bigl(f(h_{(1)} , g_{(1)}) , 1\bigl) p \bigl(f(h_{(2)} ,
g_{(2)}) , c_{(1)}\bigl) \tau(h_{(3)}g_{(3)} , 1) v(h_{(4)}g_{(4)}
, c_{(2)}) = v(h , c_{(1)}) v(g , c_{(2)})
\end{eqnarray*}
Using \equref{b} and \equref{c} we get that (LS1) holds for $v$.
Moreover from \equref{(BR3)} applied to $g = t = 1$ and $a = 1$ we
get $v(h , bc) = v(h_{(1)} , c)v(h_{(2)} , b)$ that is, (LS3) also
holds for $v$ and we proved that $v$ is indeed a $(p,f)$ - left
skew pairing on $(H, A)$. Next we apply \equref{(BR1)} for $a = b
= c = 1$:
\begin{eqnarray*}
u \bigl(f(h_{(1)} , g_{(1)}) , t_{(1)}\bigl) p \bigl(f(h_{(2)} ,
g_{(2)}) , 1\bigl) \tau(h_{(3)}g_{(3)} , t_{(2)}) v(h_{(4)}g_{(4)}
, 1) = \tau(h , t_{(1)}) \tau(g , t_{(2)})
\end{eqnarray*}
Using \equref{a} and \equref{d} we obtain (SBR1). Now from
\equref{(BR3)} applied for $a = b = c = 1$ we get:
\begin{eqnarray*}
u(1 , g_{(3)}t_{(3)}) p\bigl(1 , f(g_{(1)} , t_{(1)})\bigl) v
\bigl(h_{(2)} , f(g_{(2)} , t_{(2)})\bigl) \tau(h_{(1)} ,
g_{(4)}t_{(4)}) = \tau(h_{(1)} , t) \tau(h_{(2)} , g)
\end{eqnarray*}
From \equref{a} and \equref{c} we obtain that (SBR3) holds for
$\tau$. Furthermore, as $(A \#_{f}^{\triangleright} \, H, \sigma)$
is braided we also have from (BR5):
$$\sigma(a_{(1)} \otimes h_{(1)} , b_{(1)} \otimes
g_{(1)}) a_{(2)}(h_{(2)} \triangleright b_{(2)}) f(h_{(3)} ,
g_{(2)}) \otimes h_{(4)}g_{(3)}$$
\begin{eqnarray}
\eqlabel{(BR5)} = b_{(1)} (g_{(1)} \triangleright a_{(1)})
f(g_{(2)} , h_{(1)}) \otimes g_{(3)}h_{(2)} \sigma(a_{(2)} \otimes
h_{(3)} , b_{(2)} \otimes g_{(4)})
\end{eqnarray}
Considering $a = b = 1$ in \equref{(BR5)} we get:
\begin{eqnarray*}
\tau(h_{(1)} , g_{(1)})f(h_{(2)} , g_{(2)}) \# h_{(3)}g_{(3)} =
f(g_{(1)} , h_{(1)}) \# g_{(2)}h_{(2)} \tau(h_{(3)} , g_{(3)})
\end{eqnarray*}
Having in mind that $f$ is a coalgebra map we obtain, by applying
$\varepsilon \otimes Id$, that (SBR5) holds for $\tau$ and thus
$\tau$ is a $(u, v)$ - skew braiding on $H$. We still need to
prove that the compatibilities \equref{4.1} - \equref{4.7} hold.
Compatibilities \equref{4.1} - \equref{4.2} are obtained from
\equref{(BR5)} by considering: $a = 1$ and $g = 1$ respectively $b
= 1$ and $h = 1$ while \equref{4.3} can be derived from
\equref{(BR5)} by considering $a = b = 1$ and then applying $Id
\otimes \varepsilon$. The next two compatibilities, \equref{4.4}
and \equref{4.5}, can be obtained by letting $h = t = 1$ and $b =
1$ respectively $a = b = 1$ and $t = 1$ in \equref{(BR3)}. To this
end, relations \equref{4.6} and \equref{4.7} can be derived from
\equref{(BR1)} by considering $g = t = 1$ and $a = 1$ respectively
$a = c = 1$ and $g = 1$.

Assume now that $(A, p)$ is a braided Hopf algebra, $u$ is a
$(p,f)$ - right skew pairing on $(A, H)$, $v$ is a $(p,f)$ - left
skew pairing on $(H, A)$, $\tau$ is a $(u, v)$ - skew braiding on
$H$ and $\sigma$ is given by \equref{39} such that the
compatibilities \equref{4.1} - \equref{4.7} are fulfilled. Then,
using relations (RS2), (SBR2), (LS2) and the fact that $p$ is a
braiding we can prove that for all $a \in A$, $h \in H$ we have:
\begin{eqnarray*}
\sigma(1 \# 1 , a \# h) &{=}& u(1, h_{(1)}) p(1 , a_{(1)}) \tau(1
, h_{(2)}) v(1 , a_{(2)})\\
&{=}& \varepsilon(a)
\varepsilon(h)\\
&{=}& \varepsilon(a \# h)
\end{eqnarray*}
Moreover, using relations (RS4), (SBR4), (LS4) and again the fact
that $p$ is a braiding, we also have:
\begin{eqnarray*}
\sigma(a \# h, 1 \# 1) &{=}& u(a_{(1)} , 1) p(a_{(2)} , 1)
\tau(h_{(1)} , 1) v(h_{(2)} , 1)\\
&{=}& \varepsilon(a)
\varepsilon(h)\\
&{=}& \varepsilon(a \# h)
\end{eqnarray*}
for all $a \in A$, $h \in H$. Hence $\sigma$ also fulfills (BR4).

To prove that $\sigma$ satisfies (BR1) we start by first computing
the left hand side. Thus for all  $a, b, c \in A$ and $h, g, t \in
H$ we have:

\begin{eqnarray*}
LHS &{=}& \underline{u\bigl(a_{(1)}(h_{(1)} \triangleright
b_{(1)}) f(h_{(3)} , g_{(1)}) , t_{(1)}\bigl) p
\bigl(a_{(2)}(h_{(2)}
\triangleright b_{(2)}) f(h_{(4)} , g_{(2)}) , c_{(1)}\bigl)}\\
&& \tau(h_{(5)}g_{(3)} , t_{(2)}) v(h_{(6)}g_{(4)}, c_{(2)})\\
&\stackrel{(RS1)}{=}& u(a_{(1)} , t_{(1)}) u(h_{(1)}
\triangleright b_{(1)} , t_{(2)}) u\bigl(f(h_{(3)} , g_{(1)}) ,
t_{(3)}\bigl) p(a_{(2)} , c_{(1)})\\
&&p(h_{(2)} \triangleright b_{(2)} ,
c_{(2)})p\bigl(\underline{f(h_{(4)} , g_{(2)})} , c_{(3)}\bigl)
\tau(\underline{h_{(5)}g_{(3)}} , t_{(4)}) v(h_{(6)}g_{(4)},
c_{(4)})
\end{eqnarray*}
\begin{eqnarray*}
&\stackrel{\equref{co3}}{=}& u(a_{(1)} , t_{(1)}) u(h_{(1)}
\triangleright b_{(1)} , t_{(2)}) u\bigl(f(h_{(3)} , g_{(1)}) ,
t_{(3)}\bigl) p(a_{(2)} , c_{(1)})\\
&&p(h_{(2)} \triangleright b_{(2)} ,
c_{(2)})\underline{p\bigl(f(h_{(5)} , g_{(3)}) , c_{(3)}\bigl)}
\tau(h_{(4)}g_{(2)} , t_{(4)})
\underline{v(h_{(6)}g_{(4)}, c_{(4)})}\\
&\stackrel{(LS1)}{=}& u(a_{(1)} , t_{(1)}) u(h_{(1)}
\triangleright b_{(1)} , t_{(2)}) \underline{u\bigl(f(h_{(3)} ,
g_{(1)}) , t_{(3)}\bigl)} p(a_{(2)} , c_{(1)})\\
&&p(h_{(2)} \triangleright b_{(2)} , c_{(2)})
\underline{\tau(h_{(4)}g_{(2)} , t_{(4)})} v(h_{(5)} , c_{(3)})
v(g_{(3)} , c_{(4)})\\
&\stackrel{(SBR1)}{=}& u(a_{(1)} , t_{(1)}) u(h_{(1)}
\triangleright b_{(1)} , t_{(2)}) p(a_{(2)} , c_{(1)}) p(h_{(2)}
\triangleright b_{(2)} , c_{(2)})\\
&& \tau(h_{(3)} , t_{(3)}) \tau(g_{(1)} , t_{(4)}) v(h_{(4)} ,
c_{(3)}) v(g_{(2)} , c_{(4)})\\
&\stackrel{\equref{co1}}{=}& u(a_{(1)} , t_{(1)}) u(h_{(1)}
\triangleright b_{(1)} , t_{(2)}) p(a_{(2)} , c_{(1)})
\underline{p(h_{(3)}
\triangleright b_{(2)} , c_{(2)})}\\
&& \tau(h_{(2)} , t_{(3)}) \tau(g_{(1)} , t_{(4)})
\underline{v(h_{(4)} ,
c_{(3)})} v(g_{(2)} , c_{(4)})\\
&\stackrel{\equref{4.6}}{=}& u(a_{(1)} , t_{(1)})
\underline{u(h_{(1)} \triangleright b_{(1)} , t_{(2)})} p(a_{(2)}
, c_{(1)})
\underline{\tau(h_{(2)} , t_{(3)})} \tau(g_{(1)} , t_{(4)})\\
&& v(h_{(3)} , c_{(2)}) p(b_{(2)} , c_{(3)}) v(g_{(2)} , c_{(4)})\\
&\stackrel{\equref{4.7}}{=}& u(a_{(1)} , t_{(1)})
 p(a_{(2)}
, c_{(1)}) \tau(h_{(1)} , t_{(2)}) u(b_{(1)} , t_{(3)}) \tau(g_{(1)} , t_{(4)})\\
&& v(h_{(2)} , c_{(2)}) p(b_{(2)} , c_{(3)}) v(g_{(2)} , c_{(4)})\\
&{=}& RHS
\end{eqnarray*}
where in the second equality we also used the fact that $p$ is a
braiding. To prove (BR3) we start again by computing the left hand
side. Thus for all $a, b, c \in A$ and $h, g, t \in H$ we have:
\begin{eqnarray*}
LHS &{=}& u(a_{(1)} , g_{(5)}t_{(3)}) \underline{p\bigl(a_{(2)} ,
b_{(1)}(g_{(1)} \triangleright c_{(1)}) f(g_{(3)} ,
t_{(1)})\bigl)}
\tau(h_{(1)} , g_{(6)}t_{(4)})\\
&&\underline{v\bigl(h_{(2)} , b_{(2)}(g_{(2)} \triangleright
c_{(2)}) f(g_{(4)} , t_{(2)})\bigl)}\\
&\stackrel{(LS3)}{=}& u(a_{(1)} , \underline{g_{(5)}t_{(3)}})
p\bigl(a_{(2)} , f(g_{(3)} , t_{(1)})\bigl) p(a_{(3)} , g_{(1)}
\triangleright c_{(1)}) p(a_{(4)} , b_{(1)})\\
&& \tau(h_{(1)} , g_{(6)}t_{(4)})v\bigl(h_{(2)} ,
\underline{f(g_{(4)} , t_{(2)})}\bigl) v(h_{(3)} , g_{(2)}
\triangleright c_{(2)}) v(h_{(4)} , b_{(2)})\\
&\stackrel{\equref{co3}}{=}& u(a_{(1)} , g_{(4)}t_{(2)})
p\bigl(a_{(2)} , f(g_{(3)} , t_{(1)})\bigl) p(a_{(3)} , g_{(1)}
\triangleright c_{(1)}) p(a_{(4)} , b_{(1)})\\
&& \underline{\tau(h_{(1)} , g_{(6)}t_{(4)})v\bigl(h_{(2)} ,
f(g_{(5)} , t_{(3)})\bigl)} v(h_{(3)} , g_{(2)}
\triangleright c_{(2)}) v(h_{(4)} , b_{(2)})\\
&\stackrel{(SBR3)}{=}& \underline{u(a_{(1)} , g_{(4)}t_{(2)})
p\bigl(a_{(2)} , f(g_{(3)} , t_{(1)})\bigl)} p(a_{(3)} , g_{(1)}
\triangleright c_{(1)}) p(a_{(4)} , b_{(1)})\\
&& \tau(h_{(1)} , t_{(3)})\tau(h_{(2)} , g_{(5)}) v(h_{(3)} ,
g_{(2)}
\triangleright c_{(2)}) v(h_{(4)} , b_{(2)})\\
&\stackrel{(RS3)}{=}& u(a_{(1)} , t_{(1)}) u(a_{(2)} ,
\underline{g_{(3)}}) p(a_{(3)} , g_{(1)} \triangleright c_{(1)})
p(a_{(4)} , b_{(1)})
\tau(h_{(1)} , t_{(2)})\\
&&\tau(h_{(2)} , g_{(4)}) v(h_{(3)} , \underline{g_{(2)}
\triangleright c_{(2)}}) v(h_{(4)} , b_{(2)})\\
&\stackrel{\equref{co1}}{=}& u(a_{(1)} , t_{(1)})
\underline{u(a_{(2)} , g_{(2)}) p(a_{(3)} , g_{(1)} \triangleright
c_{(1)})} p(a_{(4)} , b_{(1)})
\tau(h_{(1)} , t_{(2)})\\
&&\tau(h_{(2)} , g_{(4)}) v(h_{(3)} , g_{(3)}
\triangleright c_{(2)}) v(h_{(4)} , b_{(2)})\\
\end{eqnarray*}
\begin{eqnarray*}
&\stackrel{\equref{4.4}}{=}& u(a_{(1)} , t_{(1)})p(a_{(2)} ,
c_{(1)})u(a_{(3)} , g_{(1)}) p(a_{(4)} , b_{(1)})
\tau(h_{(1)} , t_{(2)})\\
&&\underline{\tau(h_{(2)} , g_{(3)}) v(h_{(3)} , g_{(2)}
\triangleright c_{(2)})} v(h_{(4)} , b_{(2)})\\
&\stackrel{\equref{4.5}}{=}& u(a_{(1)} , t_{(1)})p(a_{(2)} ,
c_{(1)})u(a_{(3)} , g_{(1)}) p(a_{(4)} , b_{(1)})
\tau(h_{(1)} , t_{(2)})\\
&&v(h_{(2)} , c_{(2)})\tau(h_{(3)} , g_{(2)}) v(h_{(4)} , b_{(2)})\\
&{=}& RHS
\end{eqnarray*}
Note that in the second equality we used the fact that $p$ is a
braiding. In order to show that $\sigma$ also fulfills (BR5) we
need the following compatibilities that can be easily derived from
\equref{4.1} and \equref{4.2} by applying $\varepsilon \otimes
Id$:
\begin{eqnarray}
\eqlabel{1.2} v(h_{(1)} , b)h_{(2)} &{=}& h_{(1)}v(h_{(2)} , b)\\
\eqlabel{2.2} g_{(1)} u(a , g_{(2)}) &{=}& u(a , g_{(1)}) g_{(2)}
\end{eqnarray}
Finally, to prove (BR5) we only need to show that the following
equality holds for all $a, b \in A$ and $h, g \in H$:
\begin{eqnarray*}
u(a_{(1)} , g_{(1)})p(a_{(2)} , b_{(1)}) \tau(h_{(1)} , g_{(2)})
v(h_{(2)} , b_{(2)})a_{(3)}(h_{(3)} \triangleright
b_{(3)})f(h_{(4)} , g_{(3)}) \# h_{(5)}g_{(4)} = \\
= b_{(1)}(g_{(1)} \triangleright a_{(1)})f(g_{(2)} , h_{(1)}) \#
g_{(3)}h_{(2)}u(a_{(2)} , g_{(4)}) p(a_{(3)} , b_{(2)})
\tau(h_{(3)} , g_{(5)}) v(h_{(4)} , b_{(3)})
\end{eqnarray*}
Computing the left hand side of (BR5) we obtain:
\begin{eqnarray*}
LHS&{=}&u(a_{(1)} , g_{(1)})p(a_{(2)} , b_{(1)}) \tau(h_{(1)} ,
g_{(2)}) v(h_{(2)} , b_{(2)})a_{(3)}(h_{(3)} \triangleright
b_{(3)})\underline{f(h_{(4)} , g_{(3)})}\\
&& \# \underline{h_{(5)}g_{(4)}}\\
&\stackrel{\equref{co3}}{=}& u(a_{(1)} , g_{(1)})p(a_{(2)} ,
b_{(1)}) \tau(h_{(1)} , g_{(2)}) \underline{v(h_{(2)} ,
b_{(2)})}a_{(3)}\underline{(h_{(3)} \triangleright
b_{(3)})}f(h_{(5)} , g_{(4)})\\
&& \# \underline{h_{(4)}}g_{(3)}\\
&\stackrel{\equref{4.1}}{=}& u(a_{(1)} ,
g_{(1)})\underline{p(a_{(2)} ,
b_{(1)})} \tau(h_{(1)} , g_{(2)})\underline{a_{(3)}b_{(2)}}f(h_{(4)} , g_{(4)})\# h_{(2)}g_{(3)}v(h_{(3)} , b_{(3)})\\
&{=}& u(a_{(1)} , g_{(1)})\underline{\tau(h_{(1)} , g_{(2)})}b_{(1)}a_{(2)}p(a_{(3)} , b_{(2)})f(h_{(4)} , g_{(4)})\# \underline{h_{(2)}g_{(3)}}v(h_{(3)} , b_{(3)})\\
&\stackrel{(SBR5)}{=}& \underline{u(a_{(1)} ,
g_{(1)})}b_{(1)}\underline{a_{(2)}}p(a_{(3)} , b_{(2)})
f(h_{(4)} , g_{(4)})\#\underline{g_{(2)}}h_{(1)}\tau(h_{(2)} , g_{(3)})v(h_{(3)} , b_{(3)})\\
&\stackrel{\equref{4.2}}{=}& b_{(1)}(g_{(1)} \triangleright
a_{(1)})p(a_{(3)} , b_{(2)})
f(h_{(4)} , g_{(5)})\#\underline{g_{(2)}}h_{(1)}\underline{u(a_{(2)} , g_{(3)})}\tau(h_{(2)} , g_{(4)})\\
&&v(h_{(3)} , b_{(3)})\\
&\stackrel{\equref{2.2}}{=}& b_{(1)}(g_{(1)} \triangleright
a_{(1)})p(a_{(3)} , b_{(2)})
f(\underline{h_{(4)}} , g_{(5)})\#g_{(3)}h_{(1)}u(a_{(2)} , g_{(2)})\tau(h_{(2)} , g_{(4)})\\
&&\underline{v(h_{(3)} , b_{(3)})}\\
\end{eqnarray*}
\begin{eqnarray*}
&\stackrel{\equref{1.2}}{=}& b_{(1)}(g_{(1)} \triangleright
a_{(1)})p(a_{(3)} , b_{(2)})
\underline{f(h_{(3)} , g_{(5)})}\#g_{(3)}h_{(1)}u(a_{(2)} , g_{(2)})\underline{\tau(h_{(2)} , g_{(4)})}\\
&& v(h_{(4)} , b_{(3)})\\
&\stackrel{\equref{4.3}}{=}& b_{(1)}(g_{(1)} \triangleright
a_{(1)})p(a_{(3)} , b_{(2)})v(h_{(4)} , b_{(3)})\tau(h_{(3)} ,
g_{(5)})\underline{f(g_{(4)} , h_{(2)})}\\
&& \#\underline{g_{(3)}h_{(1)}}u(a_{(2)} , g_{(2)})\\
&\stackrel{\equref{co3}}{=}& b_{(1)}(g_{(1)} \triangleright
a_{(1)})p(a_{(3)} , b_{(2)})v(h_{(4)} , b_{(3)})\tau(h_{(3)} ,
g_{(5)})f(\underline{g_{(3)}} , h_{(1)})\\
&& \#g_{(4)}h_{(2)}\underline{u(a_{(2)} , g_{(2)})}\\
&\stackrel{\equref{2.2}}{=}& b_{(1)}(g_{(1)} \triangleright
a_{(1)})p(a_{(3)} , b_{(2)})v(h_{(4)} , b_{(3)})\tau(h_{(3)} ,
g_{(5)})f(g_{(2)} , h_{(1)})\\
&& \#\underline{g_{(4)}}h_{(2)}\underline{u(a_{(2)} , g_{(3)})}\\
&\stackrel{\equref{2.2}}{=}& b_{(1)}(g_{(1)} \triangleright
a_{(1)})p(a_{(3)} , b_{(2)})v(h_{(4)} , b_{(3)})\tau(h_{(3)} ,
g_{(5)})f(g_{(2)} , h_{(1)})\\
&& \#g_{(3)}h_{(2)}u(a_{(2)} , g_{(4})\\
&{=}&RHS
\end{eqnarray*}
In the fourth equality we used the fact that $p$ is a braiding.
Thus (BR5) holds for $\sigma$ and this ends the proof.
\end{proof}

The next Corollary gives necessary and sufficient conditions for a
smash product Hopf algebra to be braided. The result can also be
derived from \cite[Theorem 3.4]{JB2} by considering $T: H \otimes
B \rightarrow B \otimes H$ given by $T(h \otimes b) = h_{(1)}
\triangleright b \otimes h_{(2)}$, where $B$ is a left $H$-module
bialgebra via $\triangleright$.

\begin{corollary}\colabel{4.7}
Let $(A, H, \triangleright, f)$ a crossed system of Hopf algebras
such that $f$ is trivial. The following are equivalent:
\begin{enumerate}
\item[1)] $(A \#^{\triangleright} \, H, \sigma)$ is a braided
Hopf algebra\\
\item[2)] There exist four linear maps $p: A \otimes A \rightarrow
k$, $\tau: H \otimes H \rightarrow k$, $u: A \otimes H \rightarrow
k$, $v: H \otimes A \rightarrow k$ such that $(A, p)$ and $(H,
\tau)$ are braided Hopf algebras, $u$ and $v$ are skew pairings,
the compatibilities \equref{4.1} - \equref{4.2} and \equref{4.4} -
\equref{4.7} are fulfilled and the braiding $\sigma$ is given by
\equref{39}.
\end{enumerate}
\end{corollary}
\begin{proof}
It is straightforward to see that if $f$ is the trivial cocycle
then relation \equref{4.3} from \thref{4.5} is trivially fulfilled
while relations (RS3) and (LS1) which are satisfied by $u$ and $v$
collapse to (BR3) respectively (BR1).
\end{proof}

\begin{corollary}
Let $(A, H, \triangleright, f)$ be a crossed system of Hopf
algebras with $A$ a commutative Hopf algebra, $\triangleright$ the
trivial action and $(H, \tau)$ a braided Hopf algebra such that:
\begin{eqnarray*}
\tau(h_{(1)} , g_{(1)}) f(h_{(2)} , g_{(2)}) &{=}& f(g_{(1)} ,
h_{(1)}) \tau(h_{(2)} , g_{(2)})
\end{eqnarray*}
for all $h, g \in H$. Then $A \#^{\triangleright} \, H$ is a
braided Hopf algebra with the braiding given by:
$$\sigma(a \# h , b \# g) = \varepsilon(a) \varepsilon(b) \tau(h ,
g)$$ for all $h, g \in H$.
\end{corollary}
\begin{proof}
Since $A$ is commutative we can consider on $A$ the trivial
braiding given by $p = \varepsilon_{A} \otimes \varepsilon_{A}$.
Moreover $u = \varepsilon_{A} \otimes \varepsilon_{H}$ is a
$(p,f)$ - right skew pairing on $(A, H)$ and $v = \varepsilon_{H}
\otimes \varepsilon_{A}$ is a $(p, f)$ - left skew pairing on $(H,
A)$. In this case, an $(u, v)$ - skew braiding on $H$ is actually
a regular braiding and the conclusion follows by \thref{4.5}.
\end{proof}

\begin{corollary}\colabel{4.8}
Let $A$ and $H$ be Hopf algebras. The following are equivalent:
\begin{enumerate}
\item[1)] $(A \otimes H, \sigma)$ is a braided
Hopf algebra\\
\item[2)] There exist four linear maps $p: A \otimes A \rightarrow
k$, $\tau: H \otimes H \rightarrow k$, $u: A \otimes H \rightarrow
k$, $v: H \otimes A \rightarrow k$ such that $(A, p)$ and $(H,
\tau)$ are braided Hopf algebras, $u$ and $v$ are skew pairings on
$(A,\ H)$ respectively $(H,\ A)$ and the following compatibilities
are fulfilled:
\end{enumerate}
\begin{eqnarray*}
\eqlabel{t1} v(h_{(1)} , b_{(1)}) b_{(2)} \otimes h_{(2)} &{=}&
b_{(1)} v(h_{(2)} , b_{(2)}) \otimes h_{(1)}\\
\eqlabel{t2} a_{(1)} \otimes g_{(1)}u(a_{(2)} , g_{(2)}) &{=}&
a_{(2)} \otimes u(a_{(1)} , g_{(1)}) g_{(2)}\\
\eqlabel{t3} u(a_{(1)} , g) p(a_{(2)} , c) &{=}& p(a_{(1)}
, c) u(a_{(2)} , g)\\
\eqlabel{t4} \tau(h_{(1)} , g) v(h_{(2)} , c) &{=}& v(h_{(1)} , c)
\tau(h_{(2)} , g)\\
\eqlabel{t5} p(b , c_{(1)}) v(h , c_{(2)}) &{=}& v(h , c_{(1)})
p(b , c_{(2)})\\
\eqlabel{t6} u(b , t_{(1)}) \tau(h , t_{(2)}) &{=}& \tau(h ,
t_{(1)}) u(b , t_{(2)})
\end{eqnarray*}
and the braiding $\sigma$ is given by \equref{39}.
\end{corollary}
\begin{proof}
It follows from \coref{4.7} by considering $h \triangleright a =
\varepsilon(h)a$.
\end{proof}
\begin{remark}
It is easy to see that if, for example, $A$ and $H$ are both
cocommutative then the compatibilities in \coref{4.8} are
trivially fulfilled.
\end{remark}

To end with we construct some explicit braidings on the crossed
product of Hopf algebras.

\begin{examples}\exlabel{exemplulung}
$1)$ As a first example we describe, using \coref{4.8}, the
braidings on the tensor product $k[X] \otimes k[X]$, where $k[X]$
is the polynomial algebra. Recall that $k[X]$ is a commutative and
cocommutative Hopf algebra with the coalgebra structure and
antipode given by:
\begin{eqnarray*}
\Delta(X) = X \otimes 1 + 1 \otimes X, \quad \varepsilon(X) = 0,
\quad S(X) = -X
\end{eqnarray*}
Any element $\alpha \in k$, induces a skew pairing $\varphi$ on
$k[X]$ as follows:
$$\varphi(X^{i}, X^{j}) = \left \{\begin{array}{rcl}
0, \, & \mbox { if }& i \neq j\\
i!\alpha^{i}, \, & \mbox { if }& i=j
\end{array} \right.
$$
Moreover, since $k[X]$ is commutative and cocommutative the
braidings on $k[X]$ coincide with the skew pairings on $(k[X],\
k[X])$. Let $p$, $\tau$, $u$ and $v$ be the braidings induced by
$\alpha, \beta, \gamma, \tau \in k$. Then, according to
\coref{4.8} a braiding $\sigma$ on the tensor product $k[X]
\otimes k[X]$ is given by:
\begin{eqnarray*}
\sigma(X^{a} \otimes X^{b}, X^{c} \otimes X^{d}) &{=}&
\sum_{i = max\{0, a-c\}}^{min\{a, d\}}\binom{a}{i}\binom{b}{d-i}\binom{c}{a-i}\binom{d}{i}(c-a+i)!\\
&&(a-i)! (d-i)! i! \alpha^{a-i}\beta^{d-i}\gamma^{i}\tau^{c-a+i}
\end{eqnarray*}
if $a+b=c+d$ and $\sigma(X^{a} \otimes X^{b}, X^{c} \otimes X^{d})
= 0$ for $a+b \neq c+d$.\\
Moreover, all braidings on $k[X] \otimes k[X]$ arise in this way.

$2)$ In what follows $k$ is a field such that $2$ is invertible in
$k$. Let $H = k[C_{3}] = k <a \ |\ a^{3} = 1>$ be the group Hopf
algebra and $A = H_{4}$ be Sweedler's Hopf algebra. Recall that
$H_{4}$ is generated as an algebra by two elements $g$ and $x$
subject to relations:
$$
g^{2}=1, \ x^{2}=0, \ xg=-gx
$$
while the coalgebra structure and antipode are given by:
$$
\Delta(g) = g \otimes g, \ \Delta(x) = g \otimes x + x \otimes 1,
\ \varepsilon(g) = 1$$
$$\varepsilon(x) = 0, \ S(g) = g, \
S(x) = -gx.
$$
It is a straightforward computation to see that $A$ and $H$
together with the maps $\triangleright : H \otimes A \rightarrow
A$ and $f: H \otimes H \rightarrow A$ defined below is a crossed
system of Hopf algebras:
$$f(a,\ a) = f(a^{2},\ a^{2}) = g ~~{\rm and}~~ f(a^{i},\ a^{j}) = 1 ~~{\rm for}~~ (i,j) \notin \{(1,1), (2,2)\}$$
$$a \triangleright g = a^{2} \triangleright g = g,\ a \triangleright x = a^{2} \triangleright x = -x,\ a \triangleright gx = a^{2} \triangleright gx = -gx$$

In order to describe the braidings on the crossed product $H_{4}
\#_{f}^{\triangleright} \, k[C_{3}]$ we start by listing a braided
structure on $H_{4}$, left/right skew pairings and skew braidings.
For any $\alpha \in k$, $(H_{4},\ p)$ is a braided Hopf algebra,
where $p: H_{4} \otimes H_{4} \rightarrow k$ is given by:

\begin{center}
\begin{tabular} {l | r  r  r  r  }
$p$ & 1 & $g$ & $x$ & $gx$\\
\hline 1 & 1 & 1 & 0 & 0\\
$g$ & 1 & -1 & 0 & 0 \\
$x$ & 0 & 0 & $\alpha$ & $\alpha$\\
$gx$ & 0 & 0 & $\alpha$ & $\alpha$ \\
\end{tabular}
\end{center}

The linear map $u: H_{4} \otimes k[C_{3}] \rightarrow k$ defined
below is a $(p, f)$ - right skew pairing on $(H_{4}, k[C_{3}])$:
$$u(g,\ a) = u(g, \ a^{2}) = -1$$
$$u(x,\ a) = u(x,\ a^{2}) = u(gx,\ a) = u(gx,\ a^{2}) = 0$$
The linear map $v: k[C_{3}] \otimes H_{4} \rightarrow k$ defined
below is a  $(p, f)$ - left skew pairing on $(k[C_{3}], H_{4})$:
$$u(a,\ g) = u(a^{2},\ g) = -1$$
$$u(a,\ x) = u(a^{2},\ x) = u(a,\ gx) = u(a^{2},\ gx) = 0$$

For any $\gamma \in k$ such that $\gamma^{3} = 1$, the linear map
$\tau: k[C_{3}] \otimes k[C_{3}] \rightarrow k$ defined below is a
$(u, v)$ - skew braiding on $k[C_{3}]$:

\begin{center}
\begin{tabular} {l | r  r  r }
$\tau$ & 1 & $a$ & $a^{2}$\\
\hline 1 & 1 & 1 & 1\\
$a$ & 1 & $\gamma$ & $-\gamma^{2}$\\
$a^{2}$ & 1 & $-\gamma^{2}$ & $-\gamma$\\
\end{tabular}
\end{center}

Moreover, the maps $p$, $u$, $v$ and $\tau$ satisfy conditions
\equref{4.1} - \equref{4.7} in \thref{4.5}. Thus, $\sigma:
\bigl(H_{4} \#_{f}^{\triangleright} \, k[C_{3}]\bigl) \otimes
\bigl( H_{4} \#_{f}^{\triangleright} \, k[C_{3}] \bigl)
\rightarrow k$ is a braiding on the crossed product $H_{4}
\#_{f}^{\triangleright} \, k[C_{3}]$, where:
$$\sigma(b \otimes y,\ c \otimes z) = u(b_{(1)}, z_{(1)}) p(b_{(2)}, c_{(1)}) \tau(y_{(1)}, z_{(2)}) v(y_{(2)}, c_{(2)})$$
is given by:
\begin{center}
\begin{tabular} {l | r  r  r  r  r  r  }
$\sigma$\centering & 1\#1 & $1\#a$ & $1\#a^{2}$ & $g\#1$
& $g\#a$ & $g\#a^{2}$\\
\hline 1\#1\centering & 1 & 1 & 1 & 1 & 1 & 1\\
$1\#a$\centering & 1 & $\gamma$ & $-\gamma^{2}$ & $-1$ & $-\gamma$
&
$\gamma^{2}$ \\
$1\#a^{2}$\centering & 1 & $-\gamma^{2}$ & $-\gamma$ & $-1$ & $\gamma^{2}$ & $\gamma$\\
$g\#1$\centering & 1 & $-1$ & $-1$ & $-1$ & 1 & 1\\
$g\#a$\centering & 1 & $-\gamma$ & $\gamma^{2}$ & 1 & $-\gamma$ & $\gamma^{2}$\\
$g\#a^{2}$\centering & 1 & $\gamma^{2}$ & $\gamma$ & 1 &
$\gamma^{2}$ & $\gamma$\\
$x\#1$\centering & 0 & 0 & 0 & 0 & 0 & 0\\
$x\#a$\centering & 0 & 0 & 0 & 0 & 0 & 0\\
$x\#a^{2}$ \centering & 0 & 0 & 0 & 0 & 0 & 0\\
$gx\#1$ \centering & 0 & 0 & 0 & 0 & 0 & 0\\
$gx\#a$\centering & 0 & 0 & 0 & 0 & 0 & 0\\
$gx\#a^{2}$\centering & 0 & 0 & 0 & 0 & 0 & 0\\
\end{tabular}
\end{center}

\begin{center}
\begin{tabular} {l | r  r  r  r  r  r  }
$\sigma$\centering & $x\#1$ & $x\#a$ & $x\#a^{2}$ & $gx\#1$
& $gx\#a$ & $gx\#a^{2}$\\
\hline $1\#1$\centering & 0 & 0 & 0 & 0 & 0 & 0\\
$1\#a$\centering & 0 & 0 & 0 & 0 & 0 & 0\\
$1\#a^{2}$ \centering & 0 & 0 & 0 & 0 & 0 & 0\\
$g\#1$ \centering & 0 & 0 & 0 & 0 & 0 & 0\\
$g\#a$\centering & 0 & 0 & 0 & 0 & 0 & 0\\
$g\#a^{2}$\centering & 0 & 0 & 0 & 0 & 0 & 0\\
$x\#1$\centering & $\alpha$ & $-\alpha$ & $-\alpha$ & $-\alpha$ & $\alpha$ & $\alpha$\\
$x\#a$\centering & $\alpha$ & $-\alpha \gamma$ & $\alpha \gamma^{2}$ & $\alpha$ & $-\alpha \gamma$ & $\alpha \gamma^{2}$\\
$x\#a^{2}$ \centering & $\alpha$ & $\alpha \gamma^{2}$ & $\alpha \gamma$ & $\alpha$ & $\alpha \gamma^{2}$ & $\alpha \gamma$\\
$gx\#1$ \centering & $\alpha$ & $\alpha$ & $\alpha$ & $\alpha$ & $\alpha$ & $\alpha$\\
$gx\#a$\centering & $\alpha$ & $\alpha \gamma$ & $-\alpha \gamma^{2}$ & $-\alpha$ & $-\alpha \gamma$ & $\alpha \gamma^{2}$\\
$gx\#a^{2}$\centering & $\alpha$ & $-\alpha \gamma^{2}$ & $-\alpha \gamma$ & $-\alpha$ & $\alpha \gamma^{2}$ & $\alpha \gamma$\\
\end{tabular}
\end{center}

\end{examples}

\section{Acknowledgment}
The author is supported by an ''Aspirant'' Fellowship from the
Fund for Scientific Research-Flanders (Belgium) (F.W.O.
Vlaanderen). This research is part of the grant no. 88/05.10.2011
of the Romanian National Authority for Scientific Research,
CNCS-UEFISCDI.

\end{document}